\documentclass{article}

\usepackage{arxiv}

\usepackage[utf8]{inputenc} 
\usepackage[T1]{fontenc}    
\usepackage[hypertexnames=true]{hyperref}       
\usepackage{url}            
\usepackage{booktabs}       
\usepackage{amsmath, amsfonts}       
\usepackage{amsthm}
\usepackage{nicefrac}       
\usepackage{microtype}      
\usepackage{cleveref}       
\usepackage{graphicx}
\usepackage{doi}

\usepackage{amssymb}
\usepackage{algorithmic}
\usepackage[ruled,linesnumbered,vlined]{algorithm2e}
\usepackage{array}
\usepackage[caption=false,font=normalsize,labelfont=sf,textfont=sf]{subfig}
\usepackage{textcomp}
\usepackage{stfloats}
\usepackage{verbatim}
\usepackage{cite}
\usepackage{multirow} 
\usepackage{scalerel}
\usepackage{tikz}
\usepackage[title,titletoc]{appendix}

\DeclareMathOperator*{\argmin}{\arg\min}

\theoremstyle{plain}
\newtheorem{theorem}{Theorem}
\newtheorem{corollary}{Corollary}
\newtheorem{lemma}{Lemma}
\newtheorem{proposition}{Proposition}
\theoremstyle{definition}
\newtheorem{assumption}{Assumption}
\newtheorem{definition}{Definition}
\theoremstyle{remark}
\newtheorem{remark}{Remark}

\title{A Nonmonotone Front Descent Method for Bound-Constrained Multi-Objective Optimization}

\author{
	\href{https://orcid.org/0000-0002-1394-0937}{\includegraphics[scale=0.06]{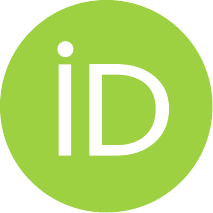}\hspace{1mm}Pierluigi Mansueto} \\
	Global Optimization Laboratory (GOL) \\
	Department of Information Engineering \\
	University of Florence \\
	Via di Santa Marta, 3, 50139, Florence, Italy \\
	\texttt{pierluigi.mansueto@unifi.it} \\
}

\renewcommand{\headeright}{Pierluigi Mansueto}
\renewcommand{\undertitle}{}
\renewcommand{\shorttitle}{}

\hypersetup{
pdftitle={A Nonmonotone Front Descent Method for Bound-Constrained Multi-Objective Optimization},
pdfsubject={},
pdfauthor={Pierluigi Mansueto},
pdfkeywords={Multi-objective optimization, Pareto front, Non-monotone line search},
}

\begin{document}
\maketitle

\begin{abstract}
	We introduce a nonmonotone extension of the Front Descent framework for multiobjective optimization. The method uses novel nonmonotone line searches that allow temporary increases in some objective functions. To our knowledge, this is the first descent algorithm employing nonmonotone strategies to generate point sets approximating the Pareto front. We establish convergence properties for the resulting sequences of sets, analogous to the original framework, and present numerical results confirming the approach’s consistency in the bound-constrained setting.
\end{abstract}

\MSCs{90C29 \and 90C30}

\section{Introduction}

In this paper, we consider multi-objective optimization problems of the form  
\begin{equation}
	\label{eq::mo_prob}
	\min_{x\in[l, u]} F(x) = (f_1(x),\ldots,f_m(x))^\top ,
\end{equation}
where $F:\mathbb{R}^n\to \mathbb{R}^m$ is continuously differentiable, and $l, u \in \mathbb{R}^n$ are lower and upper bounds defining a (possibly infinite) box constraint. We denote the feasible set as $\Omega = [l, u]$. Since multiple objectives must be minimized simultaneously, we rely on the Pareto optimality concepts; the reader is referred to \cite{EICHFELDER2021100014} for an introduction to multi-objective optimization.  

Among the main approaches to these problems that have attracted increasing attention in recent years, descent methods \cite{Fliege2000, Drummond2004, Fliege2009, Prudente2024, Lapucci2023} extend classical scalar optimization algorithms to the multi-objective setting. Originally introduced as \textit{single-point} methods, i.e., designed to approximate individual optimal solutions, they have been later generalized to compute sets of points and approximations of the Pareto frontier (see, e.g., \cite{Cocchi2020, COCCHI2021100008}), becoming competitive alternatives to scalarization \cite{pascoletti1984scalarizing, eichfelder2009adaptive} and evolutionary methods \cite{Deb2002}.  

On the other side, nonmonotone strategies, well established in scalar optimization \cite{Grippo1986, Zhang2004}, have only recently been explored in the multi-objective context \cite{Mita2019, QU2017356, Fazzio2019}. Within nonmonotone line search methods, sufficient decrease conditions are imposed relative to benchmarks (e.g., the maximum objective values obtained over the last $M > 0$ iterations). Such conditions permit temporary increases in (a subset of) the objective functions, often leading to faster overall convergence and enhanced efficiency compared to purely monotone schemes (see, e.g., \cite{Mita2019}). To date, nonmonotone variants of descent methods approximating the Pareto frontier remain unexplored. Adapting nonmonotone techniques to \textit{front-based} algorithms is indeed nontrivial, as suitable benchmarks for line searches must be defined. 

In this work, we then introduce a nonmonotone extension of the Front Descent method \cite{LAPUCCI2023242}, for which novel convergence results concerning the sequences of sets generated by the procedure have recently been established \cite{lapucci2025effectivefrontdescentalgorithmsconvergence}. We propose modifications that enable controlled increases in objective values within sets, while preserving desirable convergence properties consistent with both those proposed in \cite{lapucci2025effectivefrontdescentalgorithmsconvergence} and classical nonmonotone theory.  

The paper is organized as follows. Section \ref{sec::fd} reviews the Front Descent method and its convergence theory. Section \ref{sec::fd_nmt} introduces our novel nonmonotone variant for bound-constrained multi-objective optimization, along with a detailed analysis and convergence results. Section \ref{sec::exp} reports some numerical experiments, and Section \ref{sec::conclusions} provides concluding remarks.

\section{The Front Descent Framework}
\label{sec::fd}

The \textit{Front Descent} (\texttt{FD}) framework \cite{lapucci2025effectivefrontdescentalgorithmsconvergence} is originally designed for \textit{unconstrained} multi-objective problems. For the sake of exposition, we introduce here a variant, \textit{Front Projected Descent} (\texttt{FPD}), which incorporates modifications to handle the bounding box of problem \eqref{eq::mo_prob}. Importantly, these modifications preserve the convergence properties of the original framework.  

Given the standard partial ordering in $\mathbb{R}^m$, i.e.,
\begin{gather*}
	u\le v \iff u_j\le v_j, \;\forall\,j=1,\ldots,m,\\
	u< v \iff u_j< v_j, \;\forall\,j=1,\ldots,m,\\
	u \lneqq v\iff u\le v \land u\neq v,
\end{gather*}
\texttt{FPD} aims to find solutions $\bar{x}\in\Omega$ that are:
\begin{itemize}
	\item \textit{Pareto optimal:} $\nexists\,y\in\Omega$ s.t. $F(y)\lneqq F(\bar{x})$; 
	\item \textit{Weak Pareto optimal:} $\nexists\,y\in\Omega$ s.t. $F(y)< F(\bar{x})$;  
	\item \textit{Pareto stationary:} $\min\limits_{d\in\mathcal{D}(\bar{x})}\max\limits_{j=1,\dots,m}\nabla f_j(\bar{x})^\top  d = 0$, where $\mathcal{D}(\bar{x}) = \{v\in \mathbb{R}^n \mid \exists \bar{t}>0: \bar{x}+tv \in \Omega \ \forall t \in [0, \bar{t}\,]\}$ is the set of feasible directions at $\bar{x}$.
\end{itemize}
Pareto optimality is the strongest property, implying the others. Pareto stationarity, instead, is only a necessary condition for weak Pareto optimality, becoming sufficient under convexity assumptions. Since minimizing all objectives simultaneously is generally impossible, multiple Pareto optimal solutions (constituting the \textit{Pareto set}) usually arise, forming in the objectives space the \textit{Pareto front}. Decision makers can evaluate these trade-offs \textit{a posteriori} to select the most suitable solutions. 

To approximate the entire Pareto front, \texttt{FPD} iteratively updates a set $X^k$ of mutually nondominated solutions, i.e., no $y\in X^k$ exists such that $F(y)\lneqq F(x)$ for $x\in X^k$. The next set $X^{k+1}$ is obtained by search steps from $\hat{x}\in X^k$ along:
\begin{itemize}
	\item the \textit{projected common descent direction} \cite{Drummond2004}: 
	\begin{equation}
		\label{eq::projected_common}
		v(\hat{x}) = \argmin\limits_{\substack{d\in\mathbb{R}^n\\\hat{x}+d\in\Omega}}\max\limits_{j=1,\ldots,m} \nabla f_j(\hat{x})^\top d + \frac{1}{2}\|d\|^2;
	\end{equation}
	\item the \textit{projected partial descent directions} \cite{Drummond2004, Cocchi2020}: for subset $I\subset \{1,\ldots,m\}$, 
	\begin{equation}
		\label{eq::projected_partial}
		v^I(\hat{x}) = \argmin\limits_{\substack{d\in\mathbb{R}^n\\\hat{x}+d\in\Omega}}\;\max\limits_{j\in I}\;\nabla f_j(\hat{x})^\top d + \frac{1}{2}\|d\|^2.
	\end{equation}
\end{itemize}
Both problems admit unique solutions due to strong convexity and continuity of the objectives, and convexity of the feasible sets. Their optimal values, $\theta(\hat{x})$ and $\theta^I(\hat{x})$, are negative whenever a descent direction exists. Naturally, $v^I$ and $\theta^I$ coincide with $v$ and $\theta$, respectively, when $I = \{1,\ldots, m\}$.
Moreover, as established in \cite{Drummond2004}, $\theta$ is continuous and $v$ satisfies
\begin{equation}
	\label{eq::bound_v}
	\|v(x)\| \le 2\|J_F(x)\|,
\end{equation}
where $J_F$ is the Jacobian of $F$. A point $\bar{x}$ is Pareto stationary if and only if $\theta(\bar{x})=0$ and $v(\bar{x})=0$. Note then that it always holds that $\theta(x)\le 0$. Finally, defining $D(x,d)=\max_{j=1,\ldots,m}\nabla f_j(x)^\top d$, we have $J_F(x)v(x)\le \mathbf{1}D(x,v(x))\le \mathbf{1}\theta(x)$, with $\mathbf{1}$ being the vector of all ones in $\mathbb{R}^m$.

The \texttt{FPD} scheme is presented in Algorithm~\ref{alg::FPD}. For clarity, three key phases (steps~\ref{step::createhatXk}-\ref{step::refinearmijo}-\ref{step::createXk+1}), crucial for our nonmonotone variant, are left in a general form. In \texttt{FPD}, these are handled as follows: in step~\ref{step::createhatXk}, we set $\hat{X}^k = X^k$; in step~\ref{step::refinearmijo}, we perform the classical Armijo-type line search for multi-objective optimization \cite{Fliege2000}, i.e., $\alpha_p^k = \max_{h\in\mathbb{N}} \{\alpha_0\delta^h\mid F(x_p+\alpha_0\delta^hv(x_p))\le F(x_p)+\mathbf{1}\gamma\alpha_0\delta^hD(x_p, v(x_p))\}$; finally, in step~\ref{step::createXk+1}, we set $X^{k+1} = \hat{X}^k$.

\SetInd{1ex}{1ex}
\begin{algorithm}[!h]
	\caption{\textit{Front Projected Gradient}} 
	\label{alg::FPD}
	Input: $F:\mathbb{R}^n \rightarrow \mathbb{R}^m$, $X^0 \subset [l, u]$ set of mutually nondominated points w.r.t.\ $F$, $\alpha_0\in(0, 1],$ $\delta\in(0,1),\gamma\in(0,1)$, $\{\sigma_k\} \subseteq \mathbb{R}_+$.\\
	$k = 0$\\
	\While{a stopping criterion is not satisfied}
	{   
		Create the set $\hat{X}^k$ based on $X^k$\label{step::createhatXk}\\
		\ForAll{$x_p\in X^k$ \label{step::start_for}}
		{
			\If{$\nexists y \in \hat{X}^k \text{ s.t.\ } F(y) \lneqq F(x_p)$ \label{step::nondominance_xc}}
			{
				\If{$\theta(x_p)<-\sigma_k$\label{step::iftheta}}
				{   
					Run an Armijo-type line search to get $\alpha_p^k$\label{step::refinearmijo}
				}
				\Else{
					$\alpha_p^k=0$
				}
				$z_p^k = x_p+\alpha_p^kv(x_p)$\label{step::zck}\\
				$\hat{X}^k = (\hat{X}^k \cup \{z^k_p\}) \setminus \{y \in \hat{X}^k \mid F(z^k_p) \lneqq F(y)\}$\label{step::hatXk_1_add}\\
				\ForAll{$I\subset\{1,\ldots,m\}$ s.t.\ $\theta^I(z_p^k) < 0$\label{step::start_second_phase}}
				{
					\If{$z_p^k\in\hat{X}^k$}
					{
						$\alpha_p^I$ = $\max_{h\in\mathbb{N}} \{\alpha_0\delta^h\mid\forall y\in\hat{X}^k,\;\exists j \in \{1,\ldots,m\}: f_j(z_p^k+\alpha_0\delta^h v^I(z_p^k)) < f_j(y)\}$\label{step::2_line_search}\\
						$\hat{X}^k = (\hat{X}^k \cup \{z_p^k + \alpha_p^I v^I(z_p^k)\}) \setminus \{y \in \hat{X}^k \mid F(z_p^k + \alpha_p^I v^I(z_p^k)) \lneqq F(y)\} $\label{step::hatXk_2_add}
					}
				}
			}
		}\label{step::end_for}
		Create the set $X^{k+1}$ based on $\hat{X}^k$ \label{step::createXk+1}\\
		$k=k+1$\\
	}
	\Return{$X^k$}
\end{algorithm}

At each iteration, \texttt{FPD} updates a list of solutions $X^k$ by processing each (non-dominated) point $x_p \in X^k$ through a two-phase optimization procedure:
\begin{itemize}
	\item Phase 1 (steps \ref{step::iftheta}-\ref{step::hatXk_1_add}) acts as a refinement step, consisting of a single-point projected descent from $x_p$. This yields a new point $z_p^k$ that provides a sufficient decrease in all objective functions; if $x_p$ is already approximately Pareto-stationary (with approximation $\sigma_k$), no refinement is performed and $z_p^k = x_p$.
	\item Phase 2 (steps \ref{step::start_second_phase}-\ref{step::hatXk_2_add}) enriches $X^k$ with new nondominated points to explore the objectives space further. Starting at $z_p^k$, projected partial descent steps are taken along directions with $\theta^I(z_p^k) < 0$, as long as $z_p^k$ is not dominated. The line search considers the entire updated set $\hat{X}^k$ and must produce a point not dominated by any current solution.
\end{itemize}
Whenever a new point is added (steps \ref{step::hatXk_1_add}-\ref{step::hatXk_2_add}), dominated points are removed via a filtering procedure.

\texttt{FPD} has finite termination properties, as established in \cite[Lemmas 5.1 and 5.2]{lapucci2025effectivefrontdescentalgorithmsconvergence}. The approximation degree for Pareto stationarity is governed by a sequence $\{\sigma_k\} \subseteq \mathbb{R}_+$, which allows formally addressing both the case $\sigma_k \to 0$ and the more practical scenario where a constant $\sigma > 0$ is used for all $k$.

\begin{remark}
	By the definitions of the projected directions \eqref{eq::projected_common}-\eqref{eq::projected_partial} and of $\alpha_0 \in (0,1]$, it is straightforward to see that, if $X^0$ contains only feasible solutions, all points generated by \texttt{FPD} remain within the bounding box $[l, u]$, and all the convergence properties listed below are preserved.
\end{remark}

The convergence properties of \texttt{FPD}, concerning the sequence of sets of generated points, are based on the following two definitions and technical lemma.

\begin{definition}[{\cite[Definition 5.11]{lapucci2025effectivefrontdescentalgorithmsconvergence}}]
	Let ${\zeta}\in\mathbb{R}^m$ be a reference point and let $Y\subseteq \mathbb{R}^m$ be a (possibly infinite) set of points. We define the \textit{dominated region} as 
	$\Lambda(Y) = \left\{y\in\mathbb{R}^m\mid \exists\, \bar{y}\in Y:\; \bar{y}\le y\le {\zeta}\right\}.$
	In addition, the \textit{hyper-volume} \cite{zitzler98} associated to $Y$ is defined as the volume (or the Lebesgue measure) of the set $\Lambda(Y)$ and is denoted by $V(Y)$.
\end{definition}

It should be noted that a ``better'' approximation of the Pareto frontier is generally associated with a higher hyper-volume value. For the sake of simplicity, given $X \subset \Omega$, we will use the notation $\Lambda_F(X) = \Lambda(F(X))$ and $V_F(X) = V(F(X))$, with $F(X)$ being the image of $X$ through $F$. If $X$ only contains mutually non-dominated points, $F(x)$ is called \textit{stable set}.

\begin{lemma}[{\cite[Lemma 5.12]{lapucci2025effectivefrontdescentalgorithmsconvergence}}]
	\label{lem::5.12}
	Let $\zeta\in\mathbb{R}^m$ be a reference point and let $X$ be a set of points such that $Y = F(X)$ is a stable set and, for all $y\in Y$, $y\le \zeta$. Let $\bar{x}\in X$ and $\mu\in\Omega$ s.t.\ $F(\mu)<F(\bar{x})$. Then, the set $Z = F(X\cup\{\mu\}\setminus\{\bar{x}\})$ is such that $V(Z)-V(Y)\ge \prod_{j=1}^{m}(f_j(\bar{x})-f_j(\mu))>0.$
\end{lemma}

\begin{definition}[{\cite[Definition 5.13]{lapucci2025effectivefrontdescentalgorithmsconvergence}}]
	Let $\mathcal{X}$ be the set of all sets  $X\subseteq\Omega$ of mutually nondominated points w.r.t.\ $F$, i.e., $X\in \mathcal{X}$ if $F(X)$ is a stable set. We define the map $\Theta:\mathcal{X}\to\mathbb{R}$ as
	$\Theta(X) = \inf_{x\in X}\theta(x).$
\end{definition}

The last definition introduces a Pareto-stationarity measure for sets of mutually nondominated points. For finite sets $X$, the infimum is a minimum, so $\Theta(X)$ returns the common projected descent value $\theta$ of the ``least'' Pareto-stationary points in $X$. By construction, $\Theta(X)$ is always nonpositive and equals zero if and only if all points in $X$ are Pareto-stationary.

Before stating the convergence properties, we introduce the following assumption, which will also be essential for the theoretical analysis of our nonmonotone variant.

\begin{assumption}
	\label{ass:order}
	At each iteration $k$ of Algorithm \ref{alg::FPD}, the first point to be processed in the for loop of steps \ref{step::start_for}-\ref{step::end_for} belongs to the set $\argmin_{x\in X^k}\theta(x)$. 
\end{assumption}

\begin{proposition}[{\cite[Theorem 5.15, Corollary 5.18]{lapucci2025effectivefrontdescentalgorithmsconvergence}}]
	\label{prop:big_theta}
	Let $X^0$ be a set of mutually nondominated points and $x_0\in X^0$ be a point such that the set $\mathcal{L}(x_0) = \bigcup_{j=1}^{m}\{x\in\Omega\mid f_j(x)\le f_j(x_0)\}$ is compact.
	Let $\left\{X^k\right\}$ be the sequence of sets of nondominated points produced by \texttt{FPD} under Assumption \ref{ass:order}. 
	Then, 
	\begin{enumerate}
		\item[(i)] if $\sigma_k=\sigma>0$ for all $k$, there exists $\bar{k}$ such that $\Theta(X^{{k}})\ge -\sigma$ for all $k\ge \bar{k}$;
		\item[(ii)]  if $\sigma_k\to 0$, $\lim_{k\to \infty}\Theta(X^k) = 0$; moreover, letting $\left\{x^k\right\}$ be any sequence such that $x^k\in X^k$ for all $k$, it admits accumulation points and every accumulation point is Pareto-stationary for problem \eqref{eq::mo_prob}.
	\end{enumerate}
\end{proposition}

\begin{remark}
	As explained in \cite{lapucci2025effectivefrontdescentalgorithmsconvergence}, Assumption \ref{ass:order} prevents that \texttt{FPD} repeatedly discard points before they undergo phase 1, which would otherwise slow progress toward stationarity despite continual, but insufficient improvements. Moreover, the assumption can arguably be considered non-restrictive, as the gradients and values of $\theta$ computed at the beginning of iteration $k$ to search for the $\argmin_{x \in X^k}\theta(x)$ can be stored in memory and reused later when processing the points in $X^k$, thus keeping the extra computational cost minimal.
\end{remark}

\section{A Nonmonotone Front Descent Algorithm}
\label{sec::fd_nmt}

The monotone variant of \texttt{FPD}, which we call \texttt{FPD\_NMT}, differs from the original algorithm in steps~\ref{step::createhatXk}--\ref{step::refinearmijo}--\ref{step::createXk+1}, which are detailed in Procedures~\ref{proc::createhatXk}--\ref{proc::refinearmijo}--\ref{proc::createXk+1}, respectively. The supplementary material also provides the complete \texttt{FPD\_NMT} scheme, corresponding to Algorithm~\ref{alg::FPD} with these three procedures integrated. Unlike \texttt{FPD}, where the new set $X^{k+1}$ is generated solely based on the previous iterate set $X^k$, \texttt{FPD\_NMT} performs the search steps with respect to a new \textit{reference set}.
\begin{definition}
	\label{def::reference_set}
	Let $k$ be an iteration of Algorithm \ref{alg::FPD}. We define the \textit{reference set} $C^k$ w.r.t.\ $X^k$ as a set of solutions such that: (a) $C^k$ contains only mutually nondominated points; (b) $\forall y \in C^k$, $\exists x \in X^k$ such that $F(x) \le F(y)$; (c) $\forall x \in X^k$, $\exists y \in C^k$ such that $F(x) \le F(y)$.
\end{definition}
The reference set is thus essentially a ``worse'' set than $X^k$ in terms of the objective functions: for each point in $C^k$, there exists a point in $X^k$ whose objective function values are equal to or better than those of the point in $C^k$.

\renewcommand{\algorithmcfname}{Procedure}
\setcounter{algocf}{0}

\SetInd{1ex}{1ex}
\begin{algorithm}[!h]
	\caption{Step \ref{step::createhatXk} of Algorithm \ref{alg::FPD} for \texttt{FPD\_NMT}}
	\label{proc::createhatXk}
	$X^{l(k)} \in \argmin_{X \in \{X^k, X^{k-1},\ldots, X^{k-\min\{k, M\}}\}}V_F(X)$\label{step::Xlk}\\
	$\bar{C}^k = \{x \in X^{l(k)} \cup C^{k-1} \mid \nexists y \in X^{l(k)} \cup C^{k-1}: F(y) \lneqq F(x)\}$\label{step::hatCk}\\
	$C^k = \bar{C}^k \cup \{x \in X^k \mid \forall y \in \bar{C}^k\ \exists j: f_j(y) < f_j(x)\}$\label{step::Ck}\\
	$\hat{X}^k = C^k$\label{step::init_hatXk}\\
\end{algorithm}

Starting from $C^{-1} = \emptyset$, at each iteration $k$ of the \texttt{FPD\_NMT} algorithm, $C^k$ is created via Procedure \ref{proc::createhatXk}. In step \ref{step::Xlk}, we select the set with the lowest hyper-volume among $X^k$ and the previous $M$ sets ($M \in \mathbb{N}$ is a parameter of \texttt{FPD\_NMT}); if $k < M$, only the $k$ previous sets are considered. In order to calculate the hyper-volume metric, we use the reference point $r_p$ such that, for all $j \in \{1,\ldots, m\}$, $(r_p)_j = \max_{x \in X_{\text{all}}} f_j(x)$, with $X_{\text{all}}$ being the union of all solutions generated up to iteration $k$. In step \ref{step::hatCk}, the chosen set $X^{l(k)}$ is merged with the previous reference set $C^{k-1}$, removing the dominated solutions, to form $\bar{C}^k$. Finally, in step \ref{step::Ck}, $C^k$ is obtained by adding to $\bar{C}^k$ the points from $X^k$ that are strictly worse in at least one objective w.r.t. each point in $\bar{C}^k$. Step \ref{step::init_hatXk} initializes $\hat{X}^k$ equal to $C^k$.

\SetInd{1ex}{1ex}
\begin{algorithm}[!h]
	\caption{Step \ref{step::refinearmijo} of Algorithm \ref{alg::FPD} for \texttt{FPD\_NMT}}
	\label{proc::refinearmijo}
	\If{$\exists c \in \hat{X}^k: F(x_p) \le F(c)$ \label{step::1_line_search_start}}
	{
		$\alpha_p^k = \max_{h\in\mathbb{N}} \{\alpha_0\delta^h\mid \exists c_p^k \in \hat{X}^k: F(x_p) \le F(c_p^k) \land F(x_p+\alpha_0\delta^hv(x_p))\le F(c_p^k)+\mathbf{1}\gamma\alpha_0\delta^hD(x_p, v(x_p))\}$\label{step::1_line_search_1var}
	}
	\Else{
		$\alpha_p^k = \max_{h\in\mathbb{N}} \{\alpha_0\delta^h\mid F(x_p+\alpha_0\delta^hv(x_p))\le F(x_p)+\mathbf{1}\gamma\alpha_0\delta^hD(x_p, v(x_p))\}$ \label{step::1_line_search_end}
	}
\end{algorithm}

We now describe Procedure \ref{proc::refinearmijo}, concerning the Armijo-type line search procedure at step \ref{step::refinearmijo} of Algorithm \ref{alg::FPD}. There are two cases: (i) there exists a solution $c \in \hat{X}^k$ that is equal to or worse than the current point $x_p$; note that, since $\hat{X}^k$ is initialized equal to $C^k$, if $C^k$ is a reference set this condition always holds for the first point $x_p \in X^k$ being processed; in this case, we run step \ref{step::1_line_search_1var} of the procedure, where the goal is to find a new point that achieves sufficient decrease in all the objective functions with respect to at least one point $c_p^k$ satisfying the condition at step \ref{step::1_line_search_start} (in fact, the point $c_p^k$ may not be unique); (ii) no such solution $c \in \hat{X}^k$ exists, because it was removed during filtering; in this case, the standard Armijo-type line search is applied (step \ref{step::1_line_search_end}).

\SetInd{1ex}{1ex}
\begin{algorithm}[!h]
	\caption{Step \ref{step::createXk+1} of Algorithm \ref{alg::FPD} for \texttt{FPD\_NMT}}
	\label{proc::createXk+1}
	$\bar{X}^{l(k)} \in \argmin\limits_{X \in \{\hat{X}^k\}\cup\{X^k, X^{k-1},\ldots, X^{k - \min\{k, M-1\}}\}}V_F(X)$\label{step::barXlk}\\
	$X^{k+1} = \{x \in \hat{X}^k \cup \bar{X}^{l(k)}\mid\nexists y \in \hat{X}^k \cup \bar{X}^{l(k)}: F(y) \lneqq F(x)\}$\label{step::Xk+1}\\
\end{algorithm}

After processing all points in $X^k$, the new set $X^{k+1}$ is generated via Procedure \ref{proc::createXk+1}. Let $\bar{X}^{l(k)}$ be the set with the lowest hyper-volume value among $\hat{X}^k$ and the previous $M-1$ sets; $X^{k+1}$ then consists of all mutually non-dominated points in $\hat{X}^k \cup \bar{X}^{l(k)}$. This procedure ensures that no point in $X^{k+1}$ is dominated by any solution contained in the ``worst'' set in terms of hyper-volume.

Through the following proposition and corollary, we show that constructing $C^k$ according to Procedures \ref{proc::createhatXk}-\ref{proc::createXk+1} ensures that $C^k$ forms a reference set for $X^k$ at each iteration $k$ of the \texttt{FPD\_NMT} algorithm.

\begin{proposition}
	\label{prop::reference_set}
	Let $k$ be an iteration of \texttt{FPD\_NMT} and $X^k$ be a set of mutually non-dominated points where, $\forall y \in C^{k-1} \cup X^{l(k)}$, $\exists x \in X^{k}$ such that $F(x) \le F(y)$. Thus, $C^k$ is a reference set w.r.t. $X^k$. Moreover, $X^{k+1}$ is a set of mutually non-dominated points where, $\forall y \in C^{k} \cup \bar{X}^{l(k)}$, $\exists x \in X^{k+1}$ such that $F(x) \le F(y)$.
\end{proposition}
\begin{proof}
	First, we prove that $C^k$ satisfies the three properties defining a reference set (Definition \ref{def::reference_set}) one at a time.
	\begin{enumerate}
		\item[(a)] The set $C^k$ is created through steps \ref{step::hatCk}-\ref{step::Ck} of Procedure \ref{proc::createhatXk}. In step \ref{step::hatCk}, we get that $\bar{C}^k \subseteq X^{l(k)} \cup C^{k-1}$, with all the dominated solutions in $X^{l(k)} \cup C^{k-1}$ being filtered out. We then assume, by contradiction, that at step \ref{step::Ck} a point $\bar{x} \in X^k$ is added to $C^k$ and there exists $\bar{y} \in \bar{C}^k$ such that $F(\bar{y}) \lneqq F(\bar{x})$. Being $\bar{y} \in \bar{C}^k \subseteq X^{l(k)} \cup C^{k-1}$, we know by hypothesis that there exists $\tilde{x} \in X^k$ such that $F(\tilde{x}) \le F(\bar{y}) \lneqq F(\bar{x})$. We thus get a contradiction as $X^k$ is assumed to be a set of mutually non-dominated points.
		
		\item[(b)] By step \ref{step::Ck} of Procedure \ref{proc::createhatXk}, we distinguish two types of points: $y \in C^k \cap\bar{C}^k$ and $y \in C^k \cap X^k$. For the latter, we trivially get that $F(x) \le F(y)$, with $x = y \in X^k$. Thus, let us consider a point $y \in C^k \cap\bar{C}^k$. Since, by step \ref{step::hatCk}, $\bar{C}^k \subseteq X^{l(k)} \cup C^{k-1}$, by hypothesis there exists $x \in X^k$ such that $F(x) \le F(y)$.
		
		\item[(c)] By contradiction, we assume that there exists $x \in X^k$ such that, for all $y \in C^k$, there exists $j \in \{1,\ldots,m\}$ satisfying $f_j(y) < f_j(x)$.
		Thus, it must hold that $x \not\in C^k$, i.e., it cannot have been added at step \ref{step::Ck} of Procedure \ref{proc::createhatXk}. Therefore, it must exists $y \in C^k$ such that $F(x) \le F(y)$. This leads to a contradiction, thereby concluding the proof.
	\end{enumerate}
	We thus prove that $C^k$ is a reference set w.r.t.\ $X^k$.
	
	By step \ref{step::Xk+1} of Procedure \ref{proc::createXk+1}, we get that $X^{k+1} \subseteq \hat{X}^k \cup \bar{X}^{l(k)}$, with all the dominated points in $\hat{X}^k \cup \bar{X}^{l(k)}$ being filtered out; thus, $X^{k+1}$ is a set of mutually nondominated points. Now, we distinguish two cases: $y \in (\hat{X}^k \cup \bar{X}^{l(k)}) \cap X^{k+1}$ and $y \not\in (\hat{X}^k \cup \bar{X}^{l(k)}) \cap X^{k+1}$. In the first case, we trivially get that $F(x) \le F(y)$, with $x = y \in X^{k+1}$. In the second one, we know that there exists $x \in (\hat{X}^k \cup \bar{X}^{l(k)}) \cap X^{k+1}$ such that $F(x) \lneqq F(y)$. We can then conclude that 
	\begin{equation}
		\label{eq::Xhatk_le}
		\forall y \in \hat{X}^k \cup \bar{X}^{l(k)}, \ \exists x \in X^{k+1}: F(x) \le F(y).
	\end{equation}
	Let us now consider the set $C^k$ and $\hat{X}^k$. We observe that $\hat{X}^k$ is initialized as $C^k$ at step \ref{step::init_hatXk} of Procedure \ref{proc::createhatXk} and subsequently updated at steps~\ref{step::hatXk_1_add}–\ref{step::hatXk_2_add} of Algorithm \ref{alg::FPD}, where a filtering operation removes the solutions dominated by the newly added points. Thus, it follows that, throughout the entire iteration $k$, $\forall y \in C^k$ there exists $x \in \hat{X}^k$ such that $F(x) \le F(y)$. This result, together with \eqref{eq::Xhatk_le}, concludes the proof. 
\end{proof}

Before stating the corollary (whose proof is provided in the supplementary material), we first need to state an assumption.

\begin{assumption}
	\label{ass::same_set}
	For all iterations $k > 0$ of \texttt{FPD\_NMT}, the sets $\bar{X}^{l(k-1)}$ and $X^{l(k)}$ are equal.
\end{assumption}

\begin{corollary}
	\label{cor::induct_prop_features_C}
	Let Assumption \ref{ass::same_set} hold and $X^0$ be a set of mutually nondominated points. Then, Proposition \ref{prop::reference_set} holds for all iterations $k$ of \texttt{FPD\_NMT}.
\end{corollary}

\begin{remark}
	Assumption \ref{ass::same_set} accounts for the possible non-uniqueness of minimum arguments in Procedures \ref{proc::createhatXk} and \ref{proc::createXk+1}. It is, however, non-restrictive because: (i) a set with the worst hyper-volume in Procedure \ref{proc::createXk+1} will keep this status in the next iteration of Procedure \ref{proc::createhatXk}, since only $X_{k+1}$ (step \ref{step::Xk+1} of Procedure \ref{proc::createXk+1}) can improve its hyper-volume; (ii) analogously, one can compute the minimal argument in Procedure \ref{proc::createXk+1} and reuse it in the following iteration of Procedure \ref{proc::createhatXk}.
\end{remark}

\begin{figure*}
	\centering
	\subfloat{\includegraphics[width=0.25\textwidth]{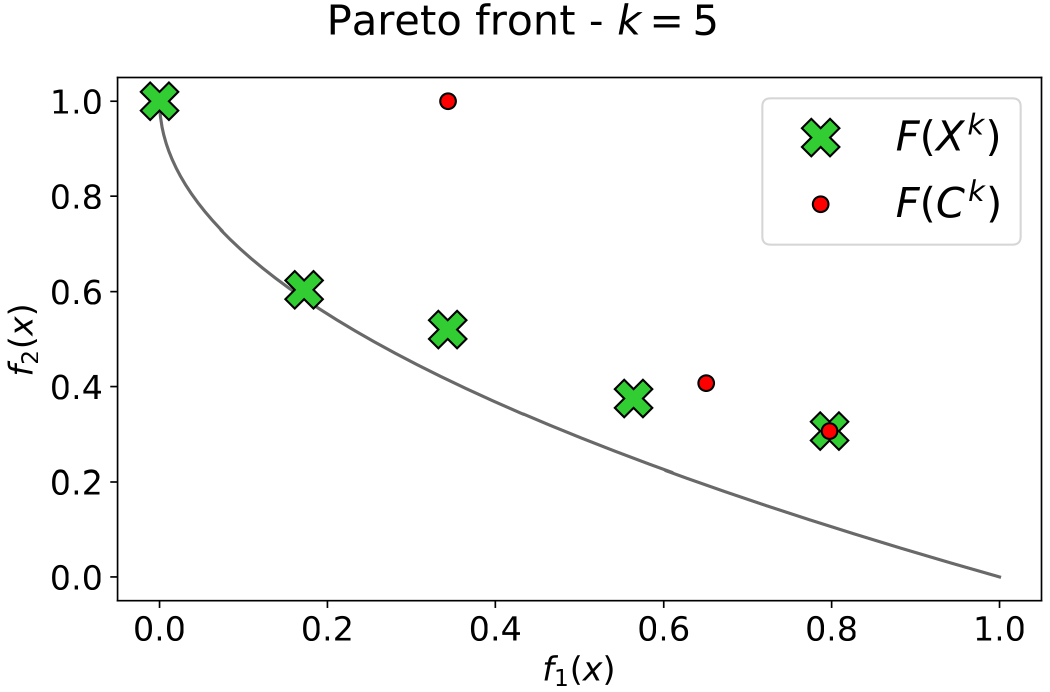}}
	\hfil
	\subfloat{\includegraphics[width=0.25\textwidth]{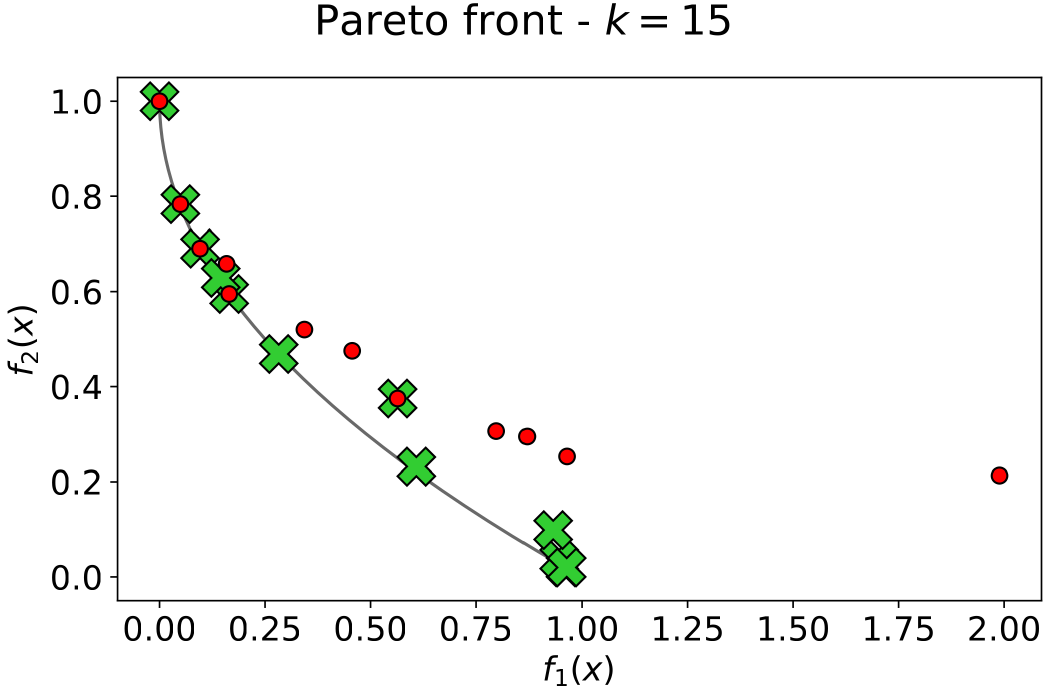}}
	\hfil
	\subfloat{\includegraphics[width=0.25\textwidth]{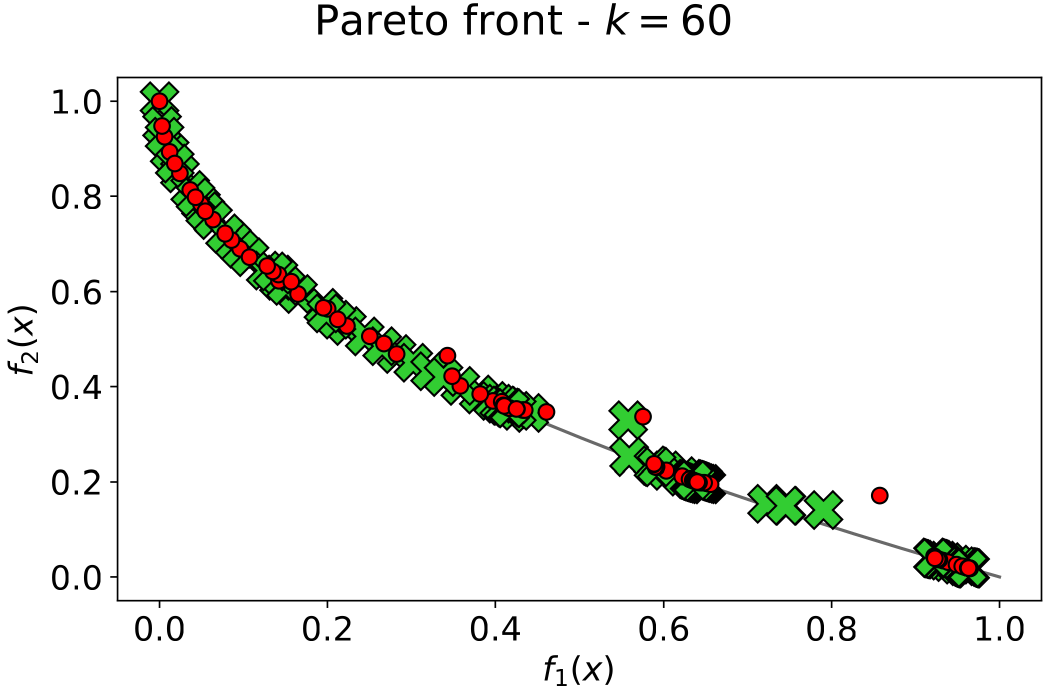}}
	\hfil
	\subfloat{\includegraphics[width=0.25\textwidth]{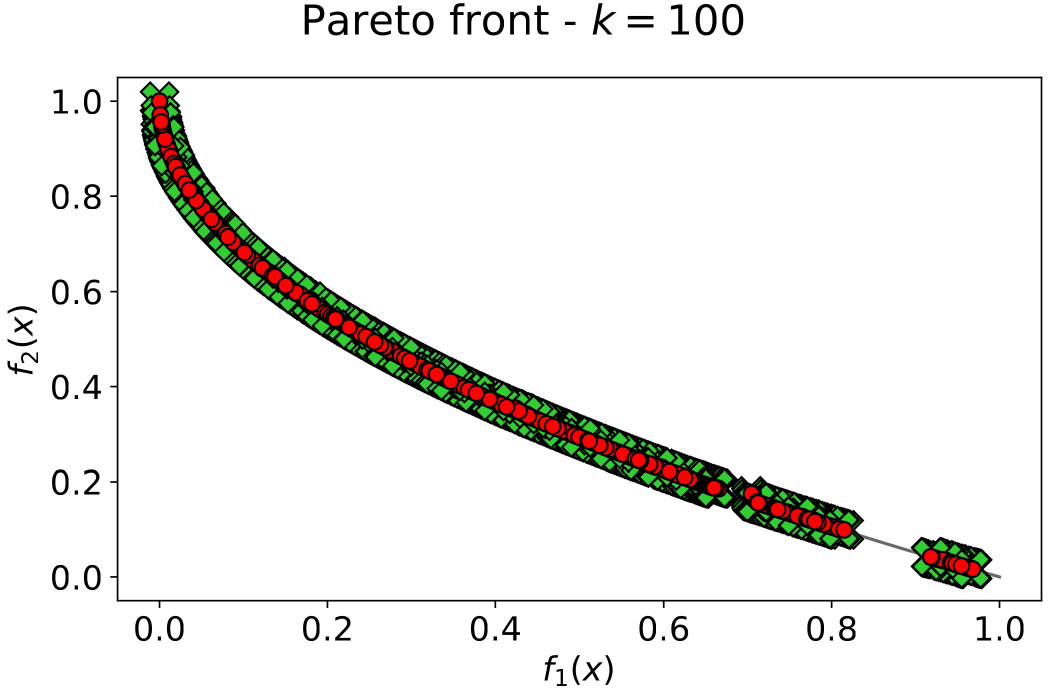}}
	\caption{Image sets of $X^k$ and $C^k$ through the mapping of $F$ for the CEC09\_2 problem \cite{zhang_multiobjective_2009} ($n=5$) obtained at iteration $k=5, 15, 60, 100$ by the \texttt{FPD\_NMT} algorithm with $M=10$.}
	\label{fig:flow}
\end{figure*}

In Figure~\ref{fig:flow}, we show the image sets of $X^k$ and $C^k$ under the mapping $F$, obtained at iterations $k = 5, 15, 60, \allowbreak 100$ for the CEC09\_2 problem \cite{zhang_multiobjective_2009} ($n=5$) using the \texttt{FPD\_NMT} algorithm with $M=10$. In the early iterations, the two image sets are distant, with $F(C^k)$ composed of points equal to or dominated by at least one point in $F(X^k)$. As the algorithm progresses, the distance between the two sets decreases. Using solutions in $C^k$ for the Armijo-type line search at step~\ref{step::1_line_search_1var} of Procedure~\ref{proc::refinearmijo} likely allows for larger steps compared to the classical procedure at step~\ref{step::1_line_search_end}, with some resulting points potentially worsening in terms of (subsets of) objective functions w.r.t.\ the current points in $X^k$. This approach is analogous to nonmonotone methods in scalar optimization: avoiding a monotone decrease of the objective functions may improve the overall convergence speed of the algorithm. Note that the line search in step~\ref{step::2_line_search} of Algorithm~\ref{alg::FPD} is also performed with respect to the solutions in $\hat{X}^k$, which is initialized as $C^k$ at each iteration $k$ of \texttt{FPD\_NMT}. Consequently, the selected step sizes are likely to be larger in this phase as well, facilitating a broader exploration of the feasible set.

We now state the finite termination properties of the algorithm and some technical lemmas that are crucial for establishing its convergence behavior. For brevity, the proofs of the following lemmas are provided in the supplementary material.

\begin{lemma}
	The line searches at steps \ref{step::1_line_search_1var}-\ref{step::1_line_search_end} of Procedure \ref{proc::refinearmijo} are well defined.
\end{lemma}
\begin{lemma}
	Let Assumption \ref{ass::same_set} hold, $X^0$ be a set of mutually nondominated points and $k$ be an iteration of \texttt{FPD\_NMT} (Algorithm \ref{alg::FPD}). The set $\hat{X}^k$ contains nondominated points at any time; thus, step \ref{step::2_line_search} is always well defined.
\end{lemma}

\begin{lemma}
	\label{lem::features_C}
	Let Assumption \ref{ass::same_set} hold, $X^0$ be a set of mutually nondominated points and $k$ be an iteration of \texttt{FPD\_NMT}. Then: (a) $V_F(C^k) \ge \max(V_F(X^{l(k)}), V_F(C^{k-1}))$; (b) $\forall y \in C^k$, $\forall \tilde{k} \ge k$, $\exists x \in X^{\tilde{k}}$ such that $F(x) \le F(y)$.
\end{lemma}

\begin{lemma}
	\label{lem::better_y}
	After step \ref{step::hatXk_1_add} of \texttt{FPD\_NMT} (Algorithm \ref{alg::FPD}), $z_p^k$ belongs to $\hat{X}^k$. Moreover, there exists $y \in X^{k+1}$ such that $F(y) \le F(z_p^k)$.
\end{lemma}

We are prepared to state the convergence properties of the algorithm.

\begin{theorem}
	\label{prop::convergence}
	Let Assumption \ref{ass::same_set} hold, $X^0$ be a set of mutually nondominated points, and $x_0 \in X^0$ be a point such that the set $\mathcal{L}(x_0) = \bigcup_{j=1}^{m}\{x \in \Omega\mid f_j(x) \le f_j(x_0)\}$ is compact. Let $\{X^k\}$ be the sequence of sets of nondominated points produced by \texttt{FPD\_NMT} under Assumption \ref{ass:order}. Then, there exists at least a subsequence $K \subseteq \{0, 1,\ldots\}$ such that
	\begin{enumerate}
		\item[(i)] if $\sigma_k = \sigma > 0$ for all $k$, there exists $\bar{k} \in K$ s.t.\ $\Theta(X^k) \ge -\sigma$ for all $k \in K$ s.t.\ $k \ge \bar{k}$;
		\item[(ii)] if $\sigma_k \to 0$, $\lim_{k \in K, k \to \infty} \Theta(X^k) = 0$.
	\end{enumerate}
\end{theorem}
\begin{proof}
	We begin the proof showing that for all $k = 0, 1,\ldots$, and for all $x \in X^k$, we have $x \in \mathcal{L}(x_0)$. Let $k$ be a generic iteration and $x \in X^k$. Since Assumption \ref{ass::same_set} holds and $X^0$ is a
	set of nondominated points, by Corollary \ref{cor::induct_prop_features_C} the set $X^k$ also consists of mutually nondominated solutions. By steps \ref{step::hatCk}-\ref{step::Ck} of Procedure \ref{proc::createhatXk}, we have that $x_0 \in C^0$; then, by point (b) of Lemma \ref{lem::features_C}, there exists a point $y_k \in X^k$ such that $F(y_k) \le F(x_0)$; since $y_k$ does not dominate $x$, there exists an index $h$ such that $f_h(x) \le f_h(y_k) \le f_h(x_0)$, i.e., $x \in \mathcal{L}(x_0)$.
	Hence, we have that
	\begin{equation}
		\label{eq::L0}
		x \in \mathcal{L}(x_0), \quad \forall x \in X^k,\ \forall k.
	\end{equation}
	
	Now, we show that a reference point $\zeta \in \mathbb{R}^m$ exists such that, for all $k = 0, 1, \ldots$, for all $x \in X^k$, $F(x) \le \zeta$. Let $\bar{\zeta} \in \mathbb{R}^m$ such that, for all $j \in \{1,\ldots,m\}$, $\bar{\zeta}_j = \max_{\bar{x}\in\mathcal{L}(x_0)}f_j(\bar{x})$; we observe that this vector is well-defined, since $F$ is continuous and $\mathcal{L}(x_0)$ is compact. Let $k$ be any iteration and $x \in X^k$. By Equation \eqref{eq::L0}, we have $x \in \mathcal{L}(x_0)$. Therefore, for all $j \in \{1,\ldots,m\}$, we have $f_j(x) \le \max_{\bar{x}\in\mathcal{L}(x_0)}f_j(\bar{x}) = \bar{\zeta}_j$. Thus, we can conclude that $F(x) \le \bar{\zeta}$, for all $x \in X^k$ and for all $k$. Moreover, by definition of the set $C^k$, we also have that
	\begin{equation}
		\label{eq::Czeta}
		F(x) \le \bar{\zeta},\quad \text{for all } x \in C^k,\ \text{for all } k.
	\end{equation}
	
	Let us consider the subsequence $\{X^{\hat{l}(k)}\}$, with $\hat{l}(k) = l(k)-1$, and, by contradiction, let us assume that the thesis is false in the two cases $\sigma_k = \sigma > 0$ and $\sigma_k \to 0$, respectively: there exists an infinite subsequence $K_1 \subseteq \{0,1,\ldots\}$ such that
	\begin{enumerate}
		\item[(i)] $\forall k \in K_1$ sufficiently large, we have $\Theta(X^{\hat{l}(k)}) < -\sigma$;
		\item[(ii)] there exists $\varepsilon > 0$ such that, $\forall k \in K_1$ sufficiently large, we have $\Theta(X^{\hat{l}(k)}) < -\varepsilon$.
	\end{enumerate}
	We can analyze both cases simultaneously, by assuming there exists $\bar{\sigma} > 0$ such that, for all $k \in K_1$ sufficiently large, $\Theta(X^{\hat{l}(k)}) < -\bar{\sigma}$. Since in case (i) we have $\bar{\sigma} = \sigma_k$, and in case (ii) $\sigma_k \to 0$, it follows that $\Theta(X^{\hat{l}(k)}) < -\sigma_{\hat{l}(k)}$ for all $k \in K_1$ sufficiently large.
	
	Now, let $\{x^{\hat{l}(k)}\}$ be the sequence of points produced by \texttt{FPD\_NMT} such that, for all $k$, $x^{\hat{l}(k)} \in X^{\hat{l}(k)}$ is the first point to be processed in the for loop of steps \ref{step::start_for}-\ref{step::end_for} of Algorithm \ref{alg::FPD} at iteration $\hat{l}(k)$. By Assumption \ref{ass:order}, $x^{\hat{l}(k)} \in \argmin_{x \in X^{\hat{l}(k)}}\theta(x)$ for all $k$, i.e., $\{x^{\hat{l}(k)}\}$ is a sequence of ``least Pareto-stationary'' points in $\{X^{\hat{l}(k)}\}$, with $\theta(x^{\hat{l}(k)}) = \Theta(X^{\hat{l}(k)})$ for all $k$. By Equation \eqref{eq::L0}, we trivially get that $\{x^{\hat{l}(k)}\}_{K_1} \subseteq\{x^{\hat{l}(k)}\}\subseteq\mathcal{L}(x_0)$ and, thus, $\{x^{\hat{l}(k)}\}_{K_1}$ admits accumulation points.
	
	Let $\bar{x}$ be an accumulation point of $\{x^{\hat{l}(k)}\}_{K_1}$, i.e., there exists $K_2 \subseteq K_1$ such that $x^{\hat{l}(k)} \to \bar{x}$ for $k \in K_2$, $k \to \infty$.
	We define $z^{\hat{l}(k)} = x^{\hat{l}(k)} + \alpha_{\hat{l}(k)}v(x^{\hat{l}(k)})$ as the point obtained at step \ref{step::zck} of Algorithm \ref{alg::FPD} while processing point $x^{\hat{l}(k)}$. Since $F$ is continuously differentiable and $\{x^{\hat{l}(k)}\}_{K_2}$ is convergent, by Equation \eqref{eq::bound_v} we have that $\{v(x^{\hat{l}(k)})\}_{K_2}$ is bounded. Moreover, $\alpha_{\hat{l}(k)} \in [0, \alpha_0]$ which is a compact set. Therefore, there exists $K_3 \subseteq K_2$ such that $v(x^{\hat{l}(k)}) \to \bar{v}$ and $\alpha_{\hat{l}(k)} \to \bar{\alpha} \in [0, \alpha_0]$ for $k \in K_3$, $k \to \infty$.
	
	By definition of $K_1$ and the sequences $\{\sigma_{\hat{l}(k)}\} \subseteq \{\sigma_k\}$ and $\{x^{\hat{l}(k)}\}$, we have that, for all $k \in K_1$ sufficiently large, $\theta(x^{\hat{l}(k)}) = \Theta(X^{\hat{l}(k)}) < -\bar{\sigma} \le -\sigma_{\hat{l}(k)} \le 0$. Also taking into account Assumption \ref{ass:order} and the instructions of the algorithm, it follows that $\alpha_{\hat{l}(k)}$ must necessarily be obtained at step \ref{step::1_line_search_1var} of Procedure \ref{proc::refinearmijo}: there exists $c^{\hat{l}(k)} \in C^{\hat{l}(k)}$ such that $F(z^{\hat{l}(k)}) \le F(c^{\hat{l}(k)}) + \mathbf{1}\gamma\alpha_{\hat{l}(k)}D(x^{\hat{l}(k)}, v(x^{\hat{l}(k)})).$
	Since $D(x^{\hat{l}(k)}, v(x^{\hat{l}(k)})) \le \theta(x^{\hat{l}(k)})$, we reformulate the last result as $F(z^{\hat{l}(k)}) \le F(c^{\hat{l}(k)}) + \mathbf{1}\gamma\alpha_{\hat{l}(k)}\theta(x^{\hat{l}(k)}).$
	By the continuity of $\theta$, we also have the following result: $\theta(\bar{x}) = \lim_{k \in K_3, k \to \infty} \theta(x^{\hat{l}(k)}) \le -\bar{\sigma} < 0$, i.e., $\bar{x}$ is not Pareto stationary. Incorporating the two last inequalities, we obtain, for $k \in K_3$ sufficiently large,
	\begin{equation}
		\label{eq::after_eta2}
		F(z^{\hat{l}(k)}) \le F(c^{\hat{l}(k)}) - \mathbf{1}\gamma\alpha_{\hat{l}(k)}\bar{\sigma}.
	\end{equation}
	
	By point (a) of Lemma \ref{lem::features_C}, the sequence $\{V_F(C^k)\}$ is monotone nondecreasing and, in fact, it admits limit $\bar{V}$. Similarly to $\bar{\zeta}$, we can define a vector $\bar{\pi} \in \mathbb{R}^m$ such that, for all $j \in \{1,\ldots,m\}$, $\bar{\pi}_j = \min_{\bar{x}\in\mathcal{L}(x_0)}f_j(\bar{x})$. It then follows that $\Lambda_F(C^k) \subseteq \{y \in \mathbb{R}^m \mid \bar{\pi} \le y \le \bar{\zeta}\}$, where the latter set is compact and therefore has a finite measure $M$. Consequently, $V_F(C^k) \le M < \infty$, and the sequence $\{V_F(C^k)\}$ must converge to a finite limit $\bar{V}$.
	
	Now, by Lemma \ref{lem::better_y} and definition of $\hat{l}(k)$, for all $k$ a point $y^{\hat{l}(k)} \in X^{l(k)}$ exists such that $F(y^{\hat{l}(k)}) \le F(z^{\hat{l}(k)})$. We then have $(X^{l(k)}\cup C^{\hat{l}(k)})\supseteq \{y^{\hat{l}(k)}\} \cup C^{\hat{l}(k)} \supseteq C^{\hat{l}(k)}$, which, by point (b) of Lemma \ref{lem::features_C} and the properties of the dominated region, implies that $\Lambda_F(X^{l(k)})=\Lambda_F(X^{l(k)}\cup C^{\hat{l}(k)})\supseteq \Lambda_F(\{y^{\hat{l}(k)}\}\cup C^{\hat{l}(k)})\supseteq \Lambda_F(C^{\hat{l}(k)}),$
	and thus
	\begin{equation}
		\label{eq:hv_CX_inc}
		V_F(X^{l(k)}) \ge V_F(\{y^{\hat{l}(k)}\} \cup C^{\hat{l}(k)}) \ge V_F(C^{\hat{l}(k)}).
	\end{equation}
	We shall also note that $F(y^{\hat{l}(k)})\le F(z^{\hat{l}(k)})<F(c^{\hat{l}(k)})$, i.e., $c^{\hat{l}(k)}$ is dominated by $y^{\hat{l}(k)}$, and then
	\begin{equation}
		\label{eq:hv_CX_same}
		V_F(\{y^{\hat{l}(k)}\} \cup C^{\hat{l}(k)}) = V_F(\{y^{\hat{l}(k)}\} \cup C^{\hat{l}(k)} \setminus \{c^{\hat{l}(k)}\}).
	\end{equation} 
	Moreover, by Corollary \ref{cor::induct_prop_features_C} and definition of reference set, for all $k$ the set $C^k$ contains mutually nondominated points, i.e., $F(C^k)$ is a stable set. Also recalling Equation \eqref{eq::Czeta}, it follows that all the assumptions of Lemma \ref{lem::5.12} are satisfied. In particular, we obtain $V_F(\{y^{\hat{l}(k)}\} \cup C^{\hat{l}(k)} \setminus \{c^{\hat{l}(k)}\})-V_F(C^{\hat{l}(k)}) \ge \prod_{j=1}^m(f_j(c^{\hat{l}(k)})-f_j(y^{\hat{l}(k)})) \ge \prod_{j=1}^m(f_j(c^{\hat{l}(k)})-f_j(z^{\hat{l}(k)})).$
	Recalling point (a) of Lemma \ref{lem::features_C}, and putting together the last result, \eqref{eq::after_eta2}-\eqref{eq:hv_CX_inc}-\eqref{eq:hv_CX_same}, we finally get that $V_F(C^k) - V_F(C^{\hat{l}(k)}) \ge V_F(X^{l(k)}) - V_F(C^{\hat{l}(k)}) \ge V_F(\{y^{\hat{l}(k)}\} \cup C^{\hat{l}(k)} \setminus \{c^{\hat{l}(k)}\}) - V_F(C^{\hat{l}(k)}) \ge (\gamma\alpha_{\hat{l}(k)}\bar{\sigma})^m$.
	Also recalling that $V_F(C^k) \to \bar{V}$, we take the limits for $k \in K_3$, $k \to \infty$ to obtain $\lim_{k \in K_3, k \to \infty} (\gamma\alpha_{\hat{l}(k)}\bar{\sigma})^m \le 0$.
	Since $\gamma > 0$ and $\bar{\sigma} > 0$, we necessarily have that $\alpha_{\hat{l}(k)} \to 0$ for $k \in K_3$, $k\to\infty$.
	
	Since $\alpha_{\hat{l}(k)}$ is defined at step \ref{step::1_line_search_1var} of Procedure \ref{proc::refinearmijo} and $\alpha_{\hat{l}(k)} \to 0$ for $k \in K_3$, $k \to \infty$, given any $q \in \mathbb{N}$ we must have $\alpha_{\hat{l}(k)} < \alpha_0\delta^q$ for all $k \in K_3$ sufficiently large. Hence, $\alpha_0\delta^q$ does not satisfy the Armijo condition: there exists $j_{\hat{l}(k)}$ such that $f_{j_{\hat{l}(k)}}(x^{\hat{l}(k)} + \alpha_0\delta^qv(x^{\hat{l}(k)})) > f_{j_{\hat{l}(k)}}(c^{\hat{l}(k)}) + \gamma\alpha_0\delta^qD(x^{\hat{l}(k)}, v(x^{\hat{l}(k)}))
	\ge f_{j_{\hat{l}(k)}}(x^{\hat{l}(k)}) + \gamma\alpha_0\delta^qD(x^{\hat{l}(k)}, v(x^{\hat{l}(k)})),$
	where the second inequality follows from the definition of $c^{\hat{l}(k)}$ at step \ref{step::1_line_search_1var} of Procedure \ref{proc::refinearmijo}.
	Taking the limits along a further subsequence such that $j_{\hat{l}(k)} = \bar{j}$, we obtain that
	$f_{\bar{j}}(\bar{x} + \alpha_0\delta^q\bar{v}) \ge f_{\bar{j}}(\bar{x}) + \gamma\alpha_0\delta^qD(\bar{x}, \bar{v}).$
	Since $q$ is arbitrary, by \cite[Lemma 4]{Fliege2000} it must be that $D(\bar{x}, \bar{v}) \ge 0$. However, we know that $D(x^{\hat{l}(k)}, v(x^{\hat{l}(k)})) \le \theta(x^{\hat{l}(k)})$ which, in the limit, turns into $D(\bar{x}, \bar{v}) \le \theta(\bar{x}) \le -\bar{\sigma} < 0$. We finally get a contradiction; hence, the proof is complete.
\end{proof}

\begin{corollary}
	Let $\{X^k\}$ be the sequence of sets of nondominated points produced by \texttt{FPD\_NMT} under the assumptions of Theorem \ref{prop::convergence} and with $\sigma_k \to 0$. Moreover, let $\{x^k\}$ be any sequence such that $x^k \in X^k$ for all $k$, and $K \subseteq \{0, 1,\ldots\}$ be the subsequence defined in Theorem \ref{prop::convergence}. Then $\{x^k\}_K$ admits accumulation points and every accumulation point is Pareto-stationary.
\end{corollary}
\begin{proof}
	By Equation \eqref{eq::L0}, we have that $\{x^k\} \subseteq \mathcal{L}(x_0)$, and so does $\{x^k\}_K$ which thus admits accumulation points. By the definition of $\Theta$, we have that, for all $k$, $0 \ge \theta(x^k) \ge \Theta(X^k)$. Thus, taking the limit along any subsequence $K_1 \subseteq K$ such that $x^k \to \bar{x}$ for $k \in K_1$, $k \to \infty$, and recalling Theorem \ref{prop::convergence}, we have $\theta(\bar{x}) = 0$, which concludes the proof.
\end{proof}

\begin{remark}
	It is worth noting that our results are established for a single subsequence. Deriving a stronger result, similar to Proposition \ref{prop:big_theta}, would require a forcing function on the distance between consecutive iterates $X^{l(k)}$ and $X^{l(k)-1}$, capturing improvement relative to the ``least Pareto-stationary'' points. However, such an assumption is rarely satisfied in practice. Future work could further investigate this direction.
\end{remark}

\begin{figure*}
	\centering
	\subfloat{\includegraphics[width=0.175\textwidth]{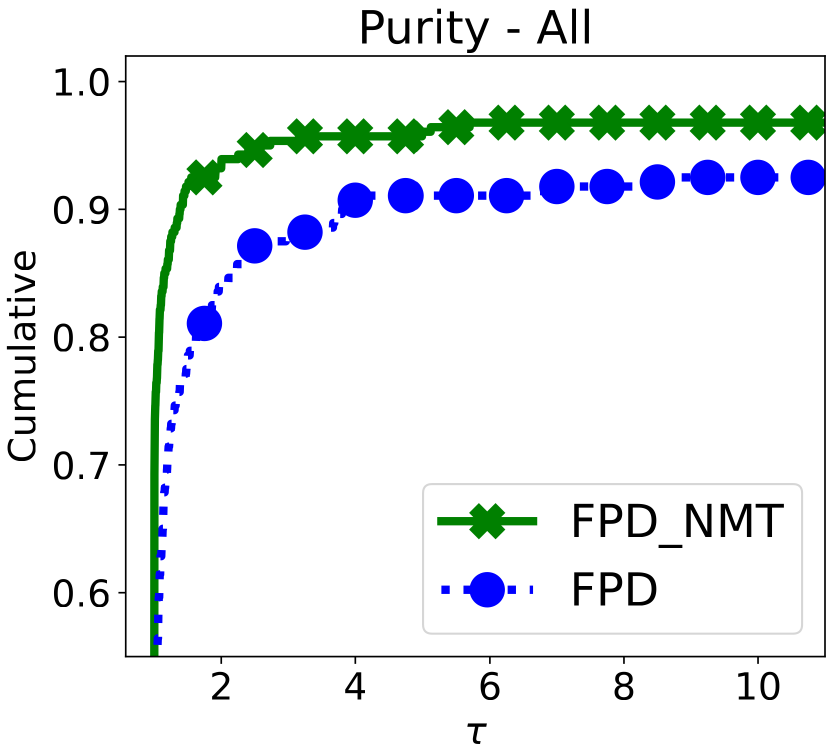}}
	\hfil
	\subfloat{\includegraphics[width=0.175\textwidth]{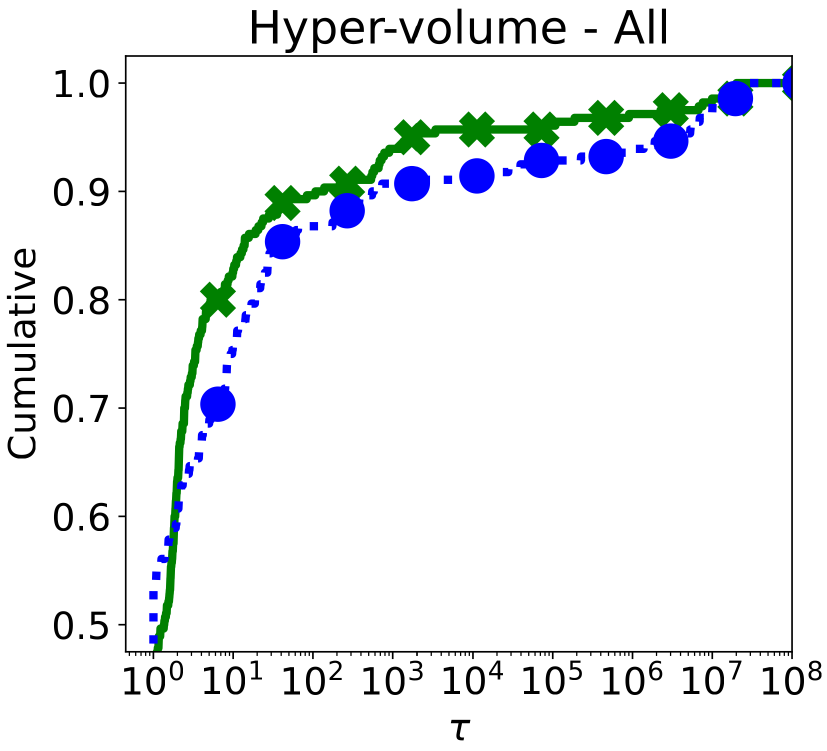}}
	\hfil
	\subfloat{\includegraphics[width=0.175\textwidth]{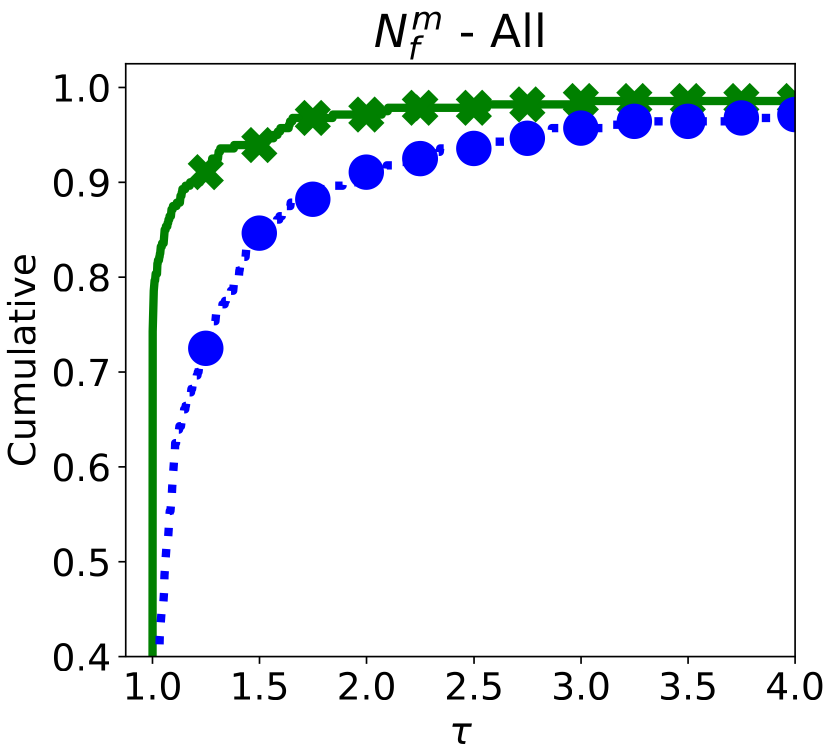}}
	\hfil
	\subfloat{\includegraphics[width=0.175\textwidth]{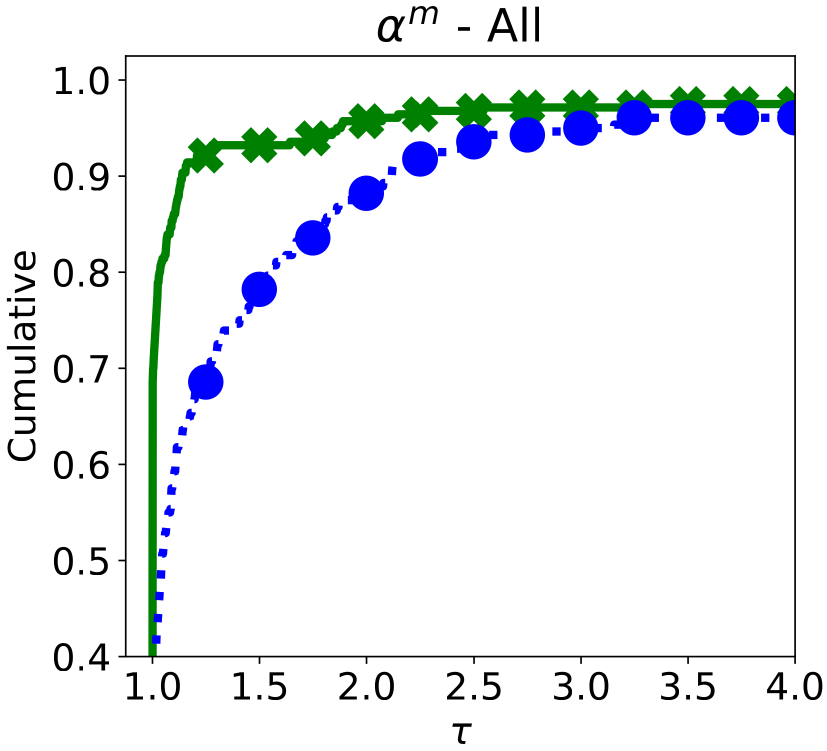}}
	\\
	\subfloat{\includegraphics[width=0.175\textwidth]{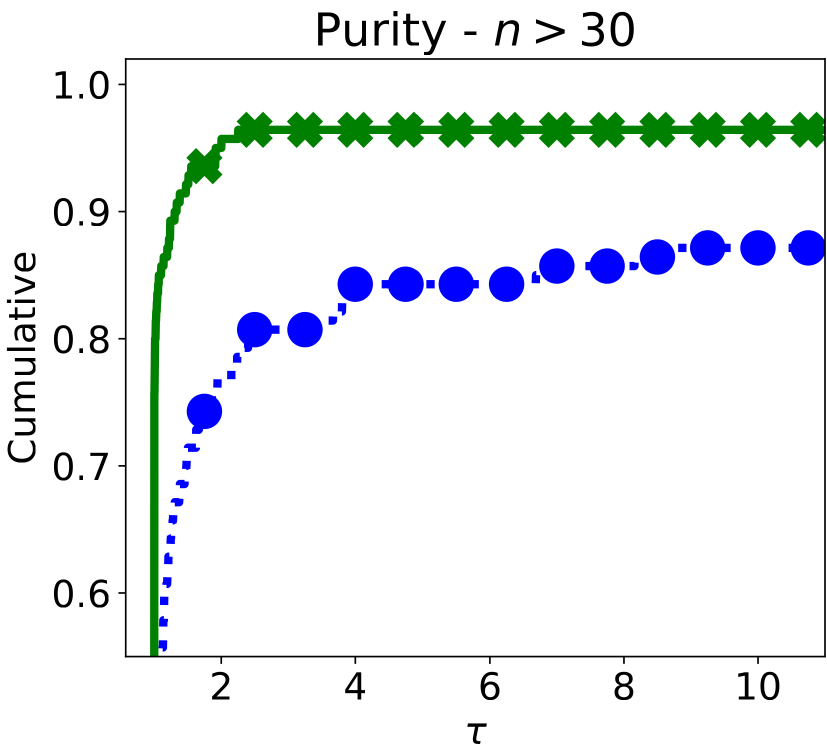}}
	\hfil
	\subfloat{\includegraphics[width=0.175\textwidth]{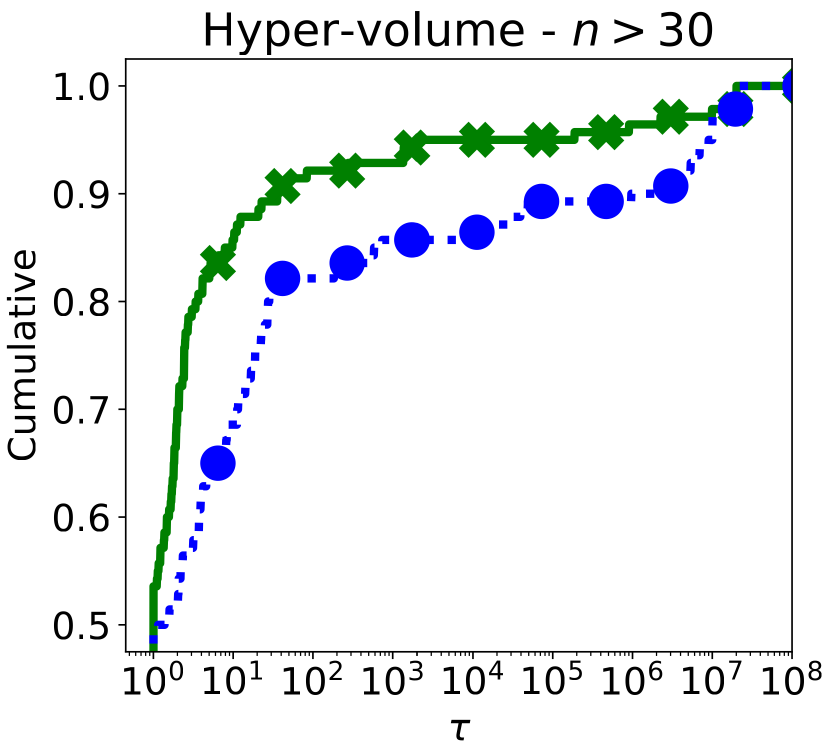}}
	\hfil
	\subfloat{\includegraphics[width=0.175\textwidth]{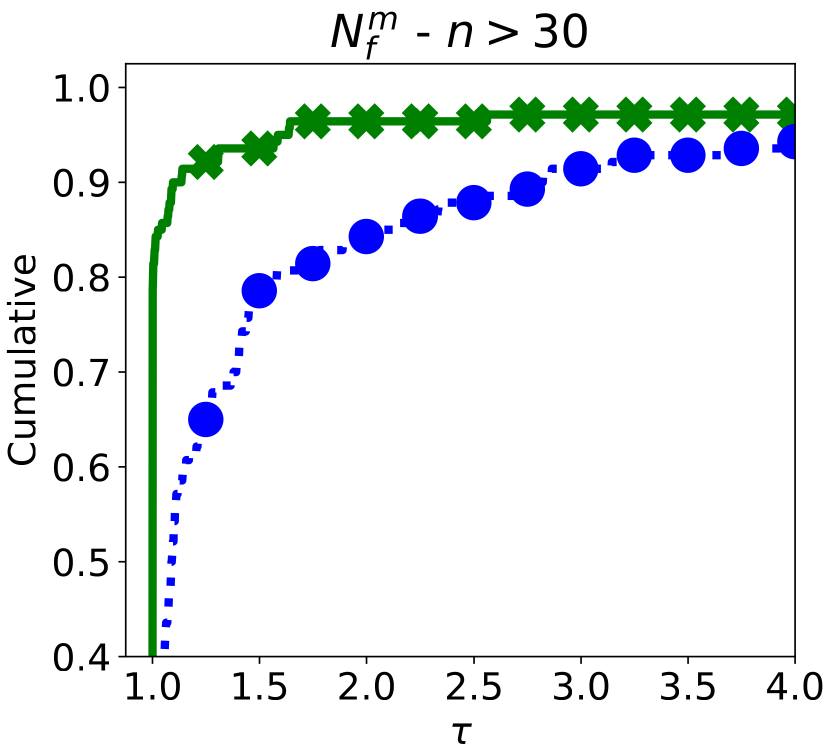}}
	\hfil
	\subfloat{\includegraphics[width=0.175\textwidth]{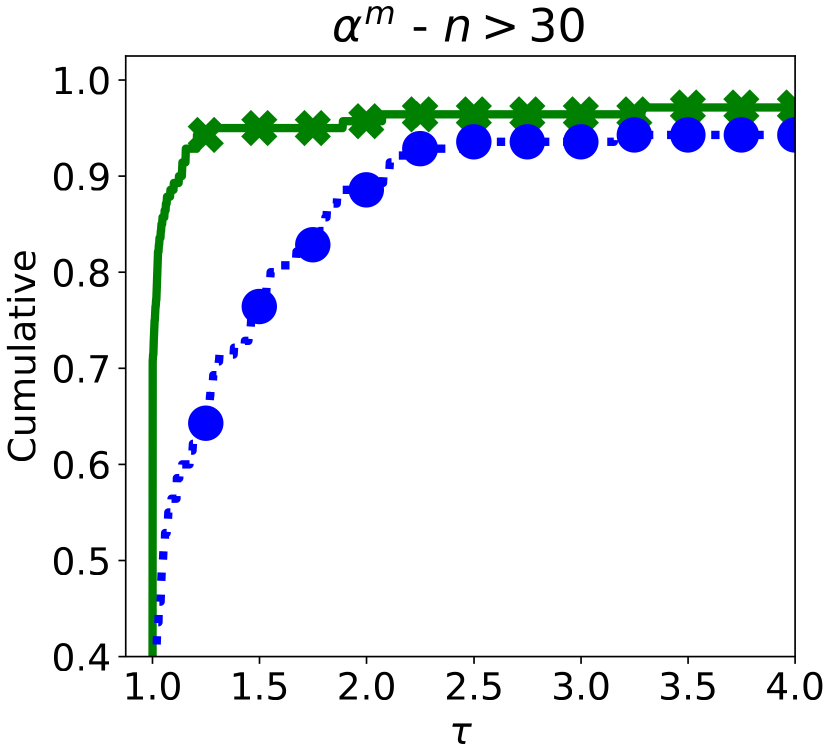}}
	\caption{Performance profiles for \texttt{FPD} and \texttt{FPD\_NMT} w.r.t.\ purity, hyper-volume, $N_f^m$, and $\alpha^m$. First row: full problem benchmark; second row: high-dimensional problems ($n > 30$).}
	\label{fig:all}
\end{figure*}

\section{Numerical results}
\label{sec::exp}

We present computational experiments assessing the effectiveness and consistency of our proposal. The Python3 code was executed on Ubuntu 22.04 with an Intel Xeon E5-2430 v2 (6 cores, 2.50 GHz) and 32~GB RAM, using Gurobi 12 \cite{gurobi} to solve problems \eqref{eq::projected_common}--\eqref{eq::projected_partial}.

We compared the nonmonotone \texttt{FPD\_NMT} to the original \texttt{FPD}, with shared parameters $\alpha_0 = 1, \delta=0.5, \gamma = 10^{-4}$, and a crowding distance \cite{Deb2002} heuristic to limit the generation of closely spaced points. For \texttt{FPD\_NMT}, the $M$ parameter was set to $4$ based on preliminary results, which we report in the supplementary material.

The benchmark is composed by box-constrained problems: CEC09 suite \cite{zhang_multiobjective_2009}, JOS\_1 ($\Omega = [0, 100]^n$) \cite{jin2001dynamic}, MAN \cite{lapucci2022memetic}, ZDT\_1 and ZDT\_3 \cite{Zitzler2000}, mostly with two objectives ($m=2$), except CEC09\_8, CEC09\_9, CEC09\_10 with $m=3$. Problem dimensions were $n \in \{5,6,8,10,12,\allowbreak15,17,20,25, \allowbreak 30,35,40, \allowbreak 45,50,100,200,300,500,1000,5000\}$. For CEC09 problems, $n$ starting points were uniformly sampled along the hyper-box diagonal; for others, single challenging initial points were considered: $(50,\ldots,50)^\top$ for JOS\_1, $(-10,\allowbreak\ldots,-10)^\top$ for MAN, and $(0.5,\ldots,0.5)^\top$ for ZDT.

Performance and robustness were assessed via \textit{performance profiles} \cite{dolan2002benchmarking}, using standard metrics -- \textit{purity} \cite{custodio11} and \textit{hyper-volume} \cite{zitzler98} -- and two ad-hoc metrics: $N_f^m$ (function evaluations per processed point) and $\alpha^m$ (average Armijo-type line-search stepsize). The reference Pareto set for purity combined the two algorithms’ solutions, discarding dominated points; the hyper-volume reference point was computed as in Section \ref{sec::fd_nmt}, with $X_\text{all}$ being the union of all solution sets. Purity, hyper-volume and $\alpha^m$ have increasing values for better solutions: thus, for consistency with performance profiles, purity and $\alpha^m$ values were inverted, while, similarly to \cite{lapucci2025effectivefrontdescentalgorithmsconvergence}, hyper-volume ones were transformed to $V_\mathcal{R} - V_\text{solver} + 10^{-7}$, where $V_\mathcal{R}$ is the hypervolume of the reference Pareto set on the considered problem instance.

In the first row of Figure \ref{fig:all}, performance profiles for \texttt{FPD} and \texttt{FPD\_NMT} are shown w.r.t.\ purity, hyper-volume, $N_f^m$, and $\alpha^m$ across the full problem benchmark. \texttt{FPD\_NMT} proved to be more effective than \texttt{FPD} in terms of purity; a similar trend has been observed for the hyper-volume metric, albeit with smaller performance differences. For $N_f^m$ and $\alpha^m$, \texttt{FPD\_NMT} clearly outperformed the original approach; improvements on one metric reasonably reflect those on the other, as the nonmonotone algorithm requires fewer line search iterations and thus fewer function evaluations. These savings, however, do not compromise our method effectiveness, as reflected in the performance profiles for purity and hyper-volume.

While the two algorithms perform similarly overall on low-dimensional problems ($n \le 30$; see the supplementary material), their differences become markedly pronounced in high-dimensional settings ($n > 30$; see the second row of Figure \ref{fig:all}), where \texttt{FPD\_NMT} clearly stands out as the superior method.

\section{Conclusions}
\label{sec::conclusions}

In this paper, we introduced a nonmonotone variant of the Front Descent framework with convergence guarantees consistent with both those of the original method \cite{lapucci2025effectivefrontdescentalgorithmsconvergence} and classical nonmonotone theory. To the best of our knowledge, this is the first attempt in the literature to integrate nonmonotone line searches into a front-based approach capable of handling sequences of point sets to approximate the Pareto front. Numerical experiments confirm the effectiveness and robustness of the proposed approach. 

Future work may include further theoretical investigations to establish stronger convergence properties, similar to those available for nonmonotone single-point multi-objective methods \cite{Mita2019}. Moreover, the algorithm could be extended to handle constraints other than box ones.

\renewcommand{\theequation}{A\arabic{equation}}
\setcounter{equation}{0}

\renewcommand{\thetable}{C\Roman{table}}
\setcounter{table}{0}
\renewcommand{\theHtable}{Supplement\thetable}

\renewcommand{\theproposition}{A\arabic{proposition}}
\setcounter{proposition}{0}
\renewcommand{\theassumption}{A\arabic{assumption}}
\setcounter{assumption}{0}

\renewcommand{\thealgocf}{A\arabic{algocf}}
\setcounter{algocf}{0}

{\appendices
	\section{Algorithmic Scheme of the Nonmonotone Front Descent Algorithm}
	
	In Algorithm \ref{alg::FPD_NMT}, we present the complete scheme of the \texttt{FPD\_NMT} algorithm, which essentially corresponds to Algorithm 1 in the paper, with Procedures 1-2-3 respectively replacing steps 4-8-17.
	
	\SetInd{1ex}{1ex}
	\begin{algorithm}[!h]
		\caption{\textit{Nonmonotone Front Projected Gradient}} 
		\label{alg::FPD_NMT}
		Input: $F:\mathbb{R}^n \rightarrow \mathbb{R}^m$, $X^0 \subset [l, u]$ set of mutually nondominated points w.r.t.\ $F$, $\alpha_0\in(0, 1],$ $\delta\in(0,1),\gamma\in(0,1)$, $\{\sigma_k\} \subseteq \mathbb{R}_+$, $M \in \mathbb{N}$.\\
		$k = 0$\\
		$C^{-1} = \emptyset$\label{stepapp::C-1}\\
		\While{a stopping criterion is not satisfied}
		{   
			$X^{l(k)} \in \argmin_{X \in \{X^k, X^{k-1},\ldots, X^{k-\min\{k, M\}}\}}V_F(X)$\label{stepapp::Xlk}\\
			$\bar{C}^k = \{x \in X^{l(k)} \cup C^{k-1} \mid \nexists y \in X^{l(k)} \cup C^{k-1}: F(y) \lneqq F(x)\}$\label{stepapp::hatCk}\\
			$C^k = \bar{C}^k \cup \{x \in X^k \mid \forall y \in \bar{C}^k\ \exists j: f_j(y) < f_j(x)\}$\label{stepapp::Ck}\\
			$\hat{X}^k = C^k$\label{stepapp::init_hatXk}\\
			\ForAll{$x_p\in X^k$ \label{stepapp::start_for}}
			{
				\If{$\nexists y \in \hat{X}^k \text{ s.t.\ } F(y) \lneqq F(x_p)$ \label{stepapp::nondominance_xc}}
				{
					\If{$\theta(x_p)<-\sigma_k$\label{stepapp::iftheta}}
					{   
						\If{$\exists c \in \hat{X}^k: F(x_p) \le F(c)$ \label{stepapp::1_line_search_start}}
						{
							$\alpha_p^k = \max_{h\in\mathbb{N}} \{\alpha_0\delta^h\mid \exists c_p^k \in \hat{X}^k: F(x_p) \le F(c_p^k) \land F(x_p+\alpha_0\delta^hv(x_p))\le F(c_p^k)+\mathbf{1}\gamma\alpha_0\delta^hD(x_p, v(x_p))\}$\label{stepapp::1_line_search_1var}
						}
						\Else{
							$\alpha_p^k = \max_{h\in\mathbb{N}} \{\alpha_0\delta^h\mid F(x_p+\alpha_0\delta^hv(x_p))\le F(x_p)+\mathbf{1}\gamma\alpha_0\delta^hD(x_p, v(x_p))\}$ \label{stepapp::1_line_search_end}
						}
					}
					\Else{
						$\alpha_p^k=0$
					}
					$z_p^k = x_p+\alpha_p^kv(x_p)$\label{stepapp::zck}\\
					$\hat{X}^k = (\hat{X}^k \cup \{z^k_p\}) \setminus \{y \in \hat{X}^k \mid F(z^k_p) \lneqq F(y)\}$\label{stepapp::hatXk_1_add}\\
					\ForAll{$I\subset\{1,\ldots,m\}$ s.t.\ $\theta^I(z_p^k) < 0$\label{stepapp::start_second_phase}}
					{
						\If{$z_p^k\in\hat{X}^k$}
						{
							$\alpha_p^I$ = $\max_{h\in\mathbb{N}} \{\alpha_0\delta^h\mid\forall y\in\hat{X}^k,\;\exists j \in \{1,\ldots,m\}: f_j(z_p^k+\alpha_0\delta^h v^I(z_p^k)) < f_j(y)\}$\label{stepapp::2_line_search}\\
							$\hat{X}^k = (\hat{X}^k \cup \{z_p^k + \alpha_p^I v^I(z_p^k)\}) \setminus \{y \in \hat{X}^k \mid F(z_p^k + \alpha_p^I v^I(z_p^k)) \lneqq F(y)\} $\label{stepapp::hatXk_2_add}
						}
					}
				}
			}\label{stepapp::end_for}
			$\bar{X}^{l(k)} \in \argmin\limits_{X \in \{\hat{X}^k\}\cup\{X^k, X^{k-1},\ldots, X^{k - \min\{k, M-1\}}\}}V_F(X)$\label{stepapp::barXlk}\\
			$X^{k+1} = \{x \in \hat{X}^k \cup \bar{X}^{l(k)}\mid\nexists y \in \hat{X}^k \cup \bar{X}^{l(k)}: F(y) \lneqq F(x)\}$\label{stepapp::Xk+1}\\
			$k=k+1$\\
		}
		\Return{$X^k$}
	\end{algorithm}
	
	\section{Supplementary Mathematical Proofs}
	
	In this section, we report mathematical proofs which did not find space in the main body of the manuscript. For convenience, wherever possible, we will use the supplementary material numbering, indicating in parentheses the corresponding one used in the main paper.
	
	\begin{corollary}[{Corollary 1 in the paper}]
		\label{corapp::induct_prop_features_C}
		Let Assumption 2 hold and $X^0$ be a set of mutually nondominated points. Then, Proposition 2 holds for all iterations $k$ of \texttt{FPD\_NMT}.
	\end{corollary}
	\begin{proof}
		The assertion follows if the assumptions of Proposition 2 hold at every iteration $k$ of the algorithm. 
		
		When $k=0$, we observe the following: by hypothesis, $X^0$ is a set of mutually nondominated points; by definition of $X^{l(k)}$ at step \ref{stepapp::Xlk} of Algorithm \ref{alg::FPD_NMT} (step 1 of Procedure 1 in the paper), $X^{l(0)} = X^0$; therefore, given that $C^{-1} = \emptyset$, $C^{-1} \cup X^{l(0)} = X^0$. Thus, all the hypotheses of Proposition 2 are satisfied for $k=0$. 
		
		The case of a generic iteration $k > 0$ straightforwardly follows from Assumption 2 and by induction using Proposition 2.
	\end{proof}
	
	\begin{lemma}[{Lemma 2 in the paper}]
		The line searches at steps \ref{stepapp::1_line_search_1var}-\ref{stepapp::1_line_search_end} of Algorithm \ref{alg::FPD_NMT} (steps 2-4 of Procedure 2 in the paper) are well defined.
	\end{lemma}
	\begin{proof}
		Since $J_F(x_p)v(x_p) \le \mathbf{1}D(x_p, v(x_p))$ and the if condition at step \ref{stepapp::iftheta} (step 7 in the paper) of Algorithm \ref{alg::FPD_NMT} ensures that $\theta(x_p)<-\sigma_k \le 0$, by \cite[Lemma 4]{Fliege2000} there exists $\bar{\alpha} > 0$ such that $F(x_p + \alpha v(x_p)) < F(x_p) + \mathbf{1}\gamma\alpha D(x_p, v(x_p))$, for all $\alpha \in (0, \bar{\alpha}]$. Thus, line search at step \ref{stepapp::1_line_search_end} of Algorithm \ref{alg::FPD_NMT} (step 4 of Procedure 2 in the paper) is eventually satisfied. As for the one at step \ref{stepapp::1_line_search_1var} (step 2 of Procedure 2 in the paper), we have, for all $\alpha \in (0, \bar{\alpha}]$, $F(x_p + \alpha v(x_p)) < F(x_p) + \mathbf{1}\gamma\alpha D(x_p, v(x_p)) \le F(c_p^k) + \mathbf{1}\gamma\alpha D(x_p, v(x_p))$, where $c_p^k$ is well-defined by the if condition at step \ref{stepapp::1_line_search_start} (step 1 of Procedure 2 in the paper). Hence, the proof is complete.
	\end{proof}
	
	\begin{lemma}[{Lemma 3 in the paper}]
		Let Assumption 2 hold, $X^0$ be a set of mutually nondominated points and $k$ be an iteration of \texttt{FPD\_NMT} (Algorithm \ref{alg::FPD_NMT}). The set $\hat{X}^k$ contains nondominated points at any time; thus, step \ref{stepapp::2_line_search} (step 15 in the paper) of Algorithm \ref{alg::FPD_NMT} is always well defined.
	\end{lemma}
	\begin{proof}
		At iteration $k$, the set $\hat{X}^k$ is initialized with the points in $C^k$, which, by Corollary \ref{corapp::induct_prop_features_C} (Corollary 1 in the paper), Proposition 2 and the definition of reference set (Definition 3), are all mutually non-dominated. The set $\hat{X}^k$ is subsequently updated exclusively at steps \ref{stepapp::hatXk_1_add} and \ref{stepapp::hatXk_2_add} (steps 12 and 16 in the paper) of Algorithm \ref{alg::FPD_NMT}. In the former step, three cases can occur: (i) $z^k_p = x_p$, which was guaranteed to be nondominated by the if-condition at step \ref{stepapp::nondominance_xc} (step 6 in the paper); (ii) $\alpha^k_p$ is defined at step \ref{stepapp::1_line_search_1var} of Algorithm \ref{alg::FPD_NMT} (step 2 of Procedure 2 in the paper); in this case, $z^k_p$ dominates $c_p^k \in \hat{X}^k$, which was nondominated as it was not filtered out by earlier executions of steps \ref{stepapp::hatXk_1_add} and \ref{stepapp::hatXk_2_add} (steps 12 and 16 in the paper); (iii) $\alpha^k_p$ is defined at step \ref{stepapp::1_line_search_end} of Algorithm \ref{alg::FPD_NMT} (step 4 of Procedure 2 in the paper); in this case, $z^k_p$ dominates $x_p$, which was again nondominated. Thus, regardless of the case, the added point $z^k_p$ is nondominated, and all the newly dominated points are removed. At step \ref{stepapp::hatXk_2_add} (step 16 in the paper), the new point $z^k_p + \alpha^I_pv^I(z^k_p)$ is nondominated by definition of $\alpha^I_p$; all the newly dominated points are once again removed. Thus, $\hat{X}^k$ always contains mutually nondominated solutions. Thus, following an identical proof as for \cite[Proposition 3.2]{LAPUCCI2023242}, we get that step \ref{stepapp::2_line_search} (step 15 in the paper) of Algorithm \ref{alg::FPD_NMT} is always well defined.
	\end{proof}
	
	\begin{lemma}[{Lemma 4 in the paper}]
		\label{lemapp::features_C}
		Let Assumption 2 hold, $X^0$ be a set of mutually nondominated points and $k$ be an iteration of \texttt{FPD\_NMT}. Then: 
		\begin{enumerate}
			\item[(a)] $V_F(C^k) \ge \max(V_F(X^{l(k)}), V_F(C^{k-1}))$;
			\item[(b)] $\forall y \in C^k$, $\forall \tilde{k} \ge k$, $\exists x \in X^{\tilde{k}}$ such that $F(x) \le F(y)$.
		\end{enumerate} 
	\end{lemma}
	\begin{proof}
		First, we prove that, for all $y \in C^{k-1}$, there exists $x \in C^k$ such that $F(x) \le F(y)$, which is a crucial property to prove the two theses. Note that the case $k=0$ is trivial by definition of $C^{-1}$ (step \ref{stepapp::C-1} of Algorithm \ref{alg::FPD_NMT})
		
		Let us consider then a point $y \in C^{k-1}$, with $k > 0$. By step \ref{stepapp::hatCk} of Algorithm \ref{alg::FPD_NMT} (step 2 of Procedure 1 in the paper), we have two cases: $y \in C^{k-1} \cap \bar{C}^k$ or $y \not\in C^{k-1} \cap \bar{C}^k$. In the first case, since, by step \ref{stepapp::Ck} of Algorithm \ref{alg::FPD_NMT} (step 3 of Procedure 1 in the paper), $\bar{C}^k \subseteq C^k$, we get that $F(x) \le F(y)$, with $x = y \in C^k$. In the second one, by step \ref{stepapp::hatCk} of Algorithm \ref{alg::FPD_NMT} (step 2 of Procedure 1 in the paper) we know  that there exists $x \in X^{l(k)} \cap \bar{C}^k$ such that $F(x) \lneqq F(y)$. Recalling again that $\bar{C}^k \subseteq C^k$, it holds that $x \in C^k$. Thus, we have proved that, for all $y \in C^{k-1}$, there exists $x \in C^k$ such that $F(x) \le F(y)$.
		
		Now, let us prove the two properties one at time.
		\begin{enumerate}
			\item[(a)] By previous result, Proposition 2 and the definition of reference set (Definition 3), we know that $(F(C^{k} \cup C^{k-1}))_{nd} = (F(C^{k}))_{nd} = F(C^{k})$. Moreover, we observe that $C^{k-1} \subseteq C^{k-1} \cup C^k$ implies that $\Lambda_F(C^{k-1}) \subseteq \Lambda_F(C^{k-1} \cup C^k)$. By the properties of the dominated region (Definition 1), we then have $\Lambda_F(C^k) = \Lambda((F(C^k \cup C^{k-1}))_\text{nd}) = \Lambda_F(C^k \cup C^{k-1}) \supseteq \Lambda_F(C^{k-1})$. Thus, we can also write that $V_F(C^k) \ge V_F(C^{k-1})$.
			Note that, following reasoning similar to that at the beginning of the proof, we can prove that, for all $y \in X^{l(k)}$, there exists $x \in C^k$ such that $F(x) \le F(y)$. Thus, using a similar argument as the one presented here, we can obtain that $V_F(C^k) \ge V_F(X^{l(k)})$. By combining this result with $V_F(C^k) \ge V_F(C^{k-1})$, we conclude the proof.
			
			\item[(b)] By Corollary \ref{corapp::induct_prop_features_C} (Corollary 1 in the paper) and definition of the reference set, for all $y \in C^{\tilde{k}}$, there exists $x \in X^{\tilde{k}}$ such that $F(x) \le F(y)$. If $\tilde{k} = k$, the proof is thus complete; otherwise, by applying an argument analogous to that used at the end of the proof of Proposition 2, we can combine this result with the property established at the beginning of the proof and proceed inductively to obtain the claim.
		\end{enumerate}
	\end{proof}
	
	\begin{lemma}[{Lemma 5 in the paper}]
		\label{lemapp::better_y}
		After step \ref{stepapp::hatXk_1_add} (step 12 in the paper) of \texttt{FPD\_NMT} (Algorithm \ref{alg::FPD_NMT}), $z_p^k$ belongs to $\hat{X}^k$. Moreover, there exists $y \in X^{k+1}$ such that $F(y) \le F(z_p^k)$.
	\end{lemma}
	\begin{proof}
		The result follows as in \cite[Lemma 3.1]{LAPUCCI2023242} with $\tilde{k} = k+1$, recalling that the set $X^{k+1}$ is the result of repeated application of steps \ref{stepapp::hatXk_1_add} and \ref{stepapp::hatXk_2_add} (steps 12 and 16 in the paper) and of the execution of step \ref{stepapp::Xk+1} (step 2 of Procedure 3 in the paper), starting from $\hat{X}^k$ at some point when $z^k_p \in \hat{X}^k$.
	\end{proof}
	
	\section{Supplementary Numerical Results}
	
	In Figure \ref{fig:preliminary}, we report the performance profiles w.r.t.\ purity, hyper-volume, $N_f^m$ and $\alpha^m$ obtained by \texttt{FPD\_NMT} with $M \in \{2, 4, 20\}$ on a subset of problems: CEC09\_1, CEC09\_3 \cite{zhang_multiobjective_2009}, JOS\_1 \cite{jin2001dynamic}, MAN \cite{lapucci2022memetic}, ZDT\_1 and ZDT\_3 \cite{Zitzler2000} with $n \in \{5, 50, 500, 5000\}$. Additional details on the experimental settings are provided in Section 4 of the paper. 
	
	\texttt{FPD\_NMT} with $M=4$ clearly outperformed the other variants in terms of purity and hyper-volume. For $N_f^m$, the three methods showed similar performance, while the performance profiles for $\alpha^m$ indicate a clear advantage for the variant with $M=20$. This aligns with the expectations: increasing the memory parameter $M$ makes the sufficient decrease condition in the nonmonotone line search at step \ref{stepapp::1_line_search_1var} of Algorithm \ref{alg::FPD_NMT} (step 2 of Procedure 2 in the paper) easier to satisfy with fewer iterations, leading to larger step sizes. A similar trend was observed in the line search at step \ref{stepapp::2_line_search} (step 15 in the paper) of Algorithm \ref{alg::FPD_NMT}. However, this behavior does not translate into superior performance for purity and hyper-volume. 
	
	Notably, \texttt{FPD\_NMT} with $M=4$ remained the second most robust approach with respect to $\alpha^m$. Its overall performance motivated our choice of this variant for the comparisons presented in the paper.
	
	\begin{figure*}
		\centering
		\subfloat{\includegraphics[width=0.25\textwidth]{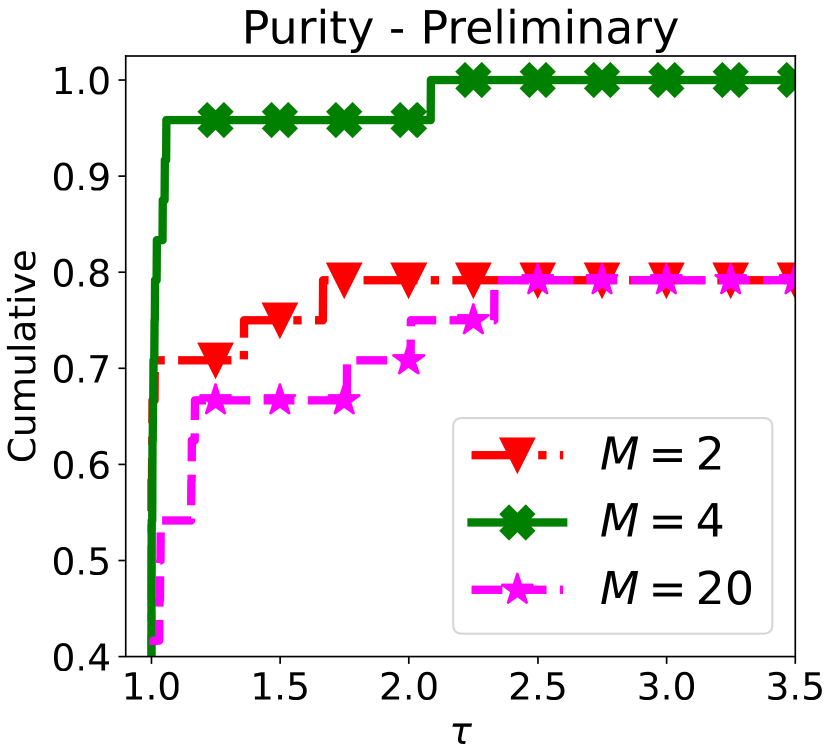}}
		\hfil
		\subfloat{\includegraphics[width=0.25\textwidth]{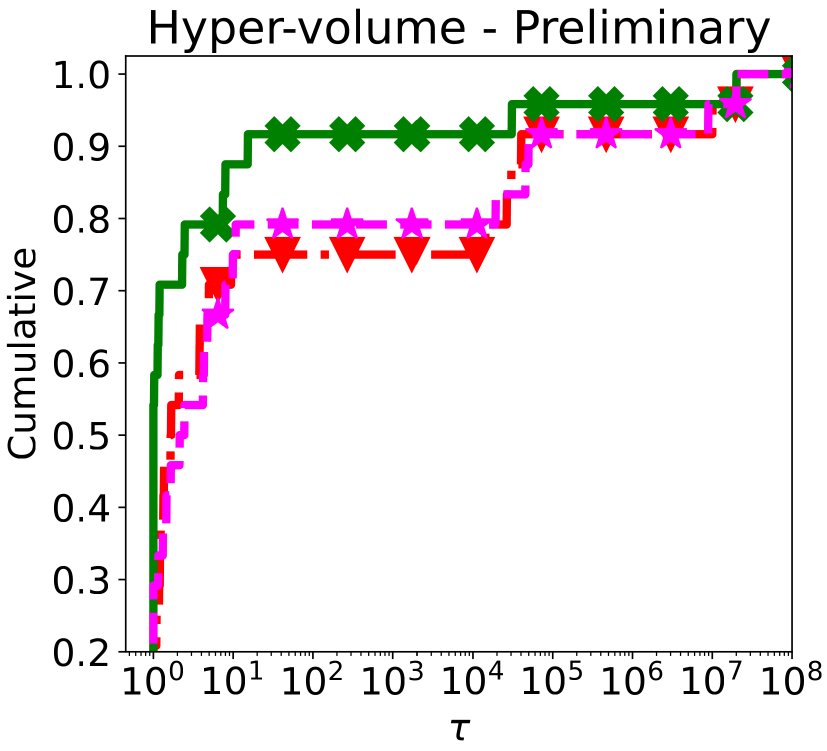}}
		\hfil
		\subfloat{\includegraphics[width=0.25\textwidth]{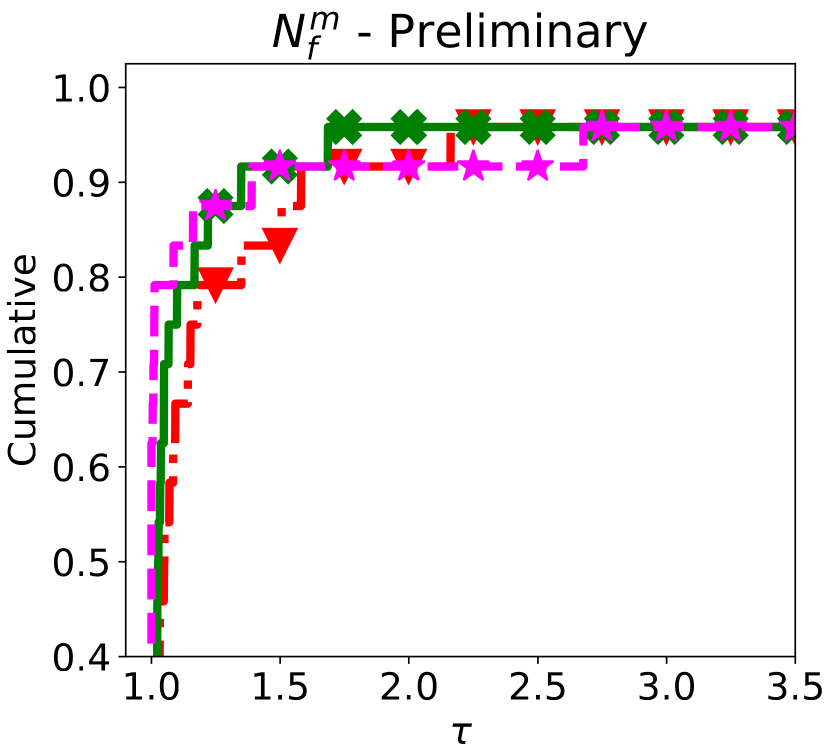}}
		\hfil
		\subfloat{\includegraphics[width=0.25\textwidth]{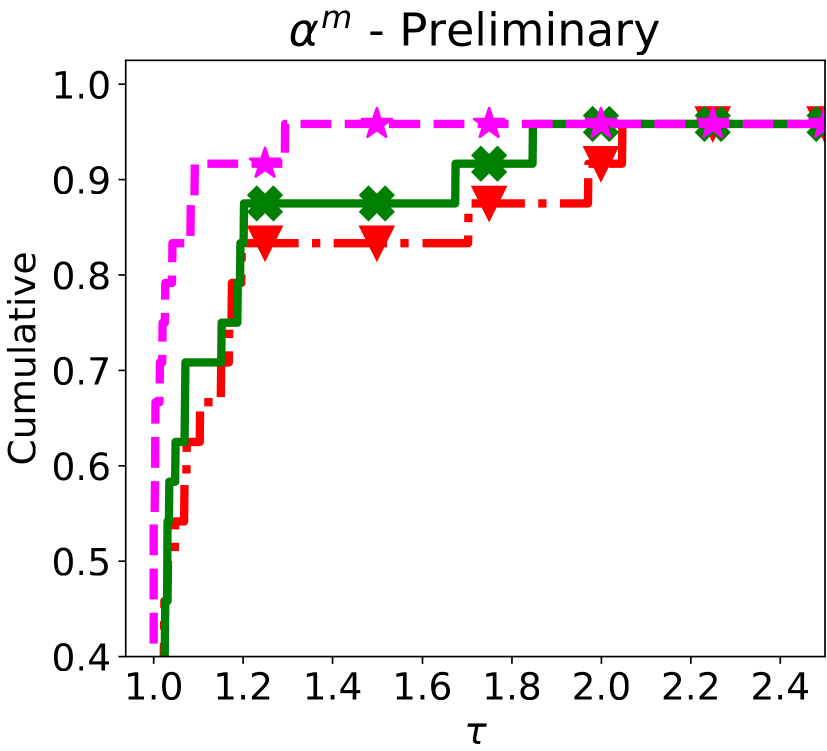}}
		\caption{Performance profiles w.r.t.\ purity, hyper-volume, $N_f^m$ and $\alpha^m$ obtained by \texttt{FPD\_NMT} with $M \in \{2, 4, 20\}$ on a subset of problems.}
		\label{fig:preliminary}
	\end{figure*}
	
	In Figure \ref{figapp:all}, we show the performance profiles for
	\texttt{FPD} and \texttt{FPD\_NMT} w.r.t. purity, hyper-volume, $N_f^m$ and $\alpha^m$ across the full problem benchmark (first row), the low-dimensional problems ($n \le 30$; second row) and the high-dimensional problems ($n > 30$; third row). The results for the full benchmark and the high-dimensional problems were already presented in Figure 2 of the paper; we include them here again for a more direct comparison with the performance profiles on the low-dimensional problems. A discussion of these results is provided in Section 4 of the paper.
	
	\begin{figure*}
		\centering
		\subfloat{\includegraphics[width=0.25\textwidth]{Purity_all.pdf}}
		\hfil
		\subfloat{\includegraphics[width=0.25\textwidth]{HV_all.pdf}}
		\hfil
		\subfloat{\includegraphics[width=0.25\textwidth]{Nf_mean_all.pdf}}
		\hfil
		\subfloat{\includegraphics[width=0.25\textwidth]{Alpha_mean_all.pdf}}
		\\
		\subfloat{\includegraphics[width=0.25\textwidth]{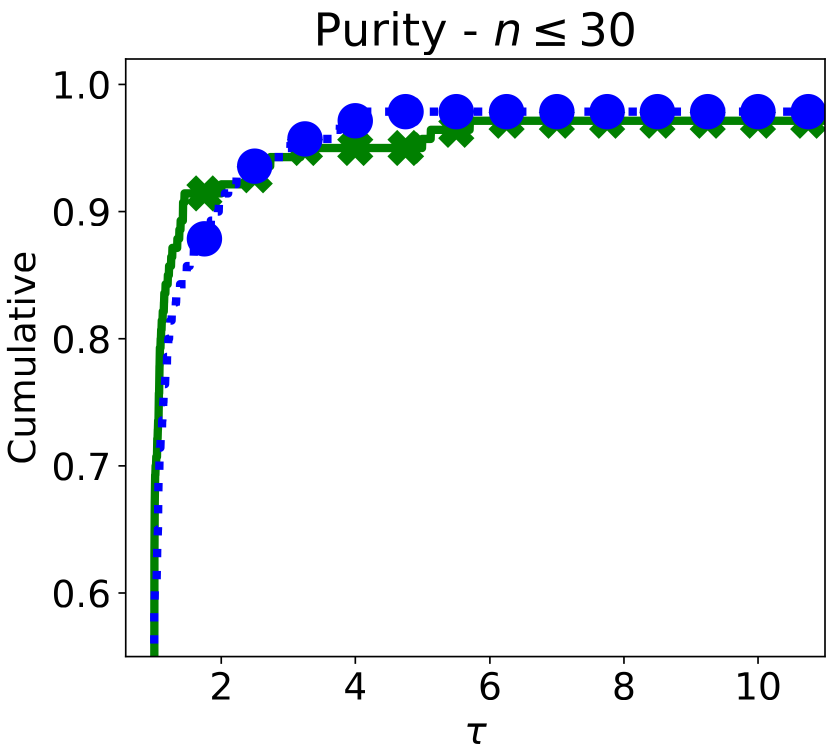}}
		\hfil
		\subfloat{\includegraphics[width=0.25\textwidth]{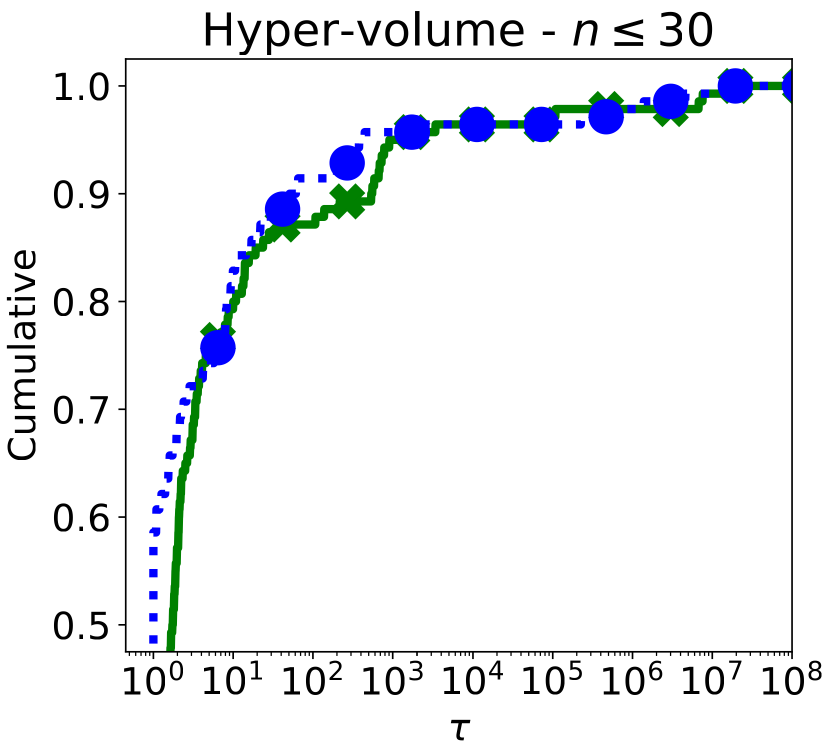}}
		\hfil
		\subfloat{\includegraphics[width=0.25\textwidth]{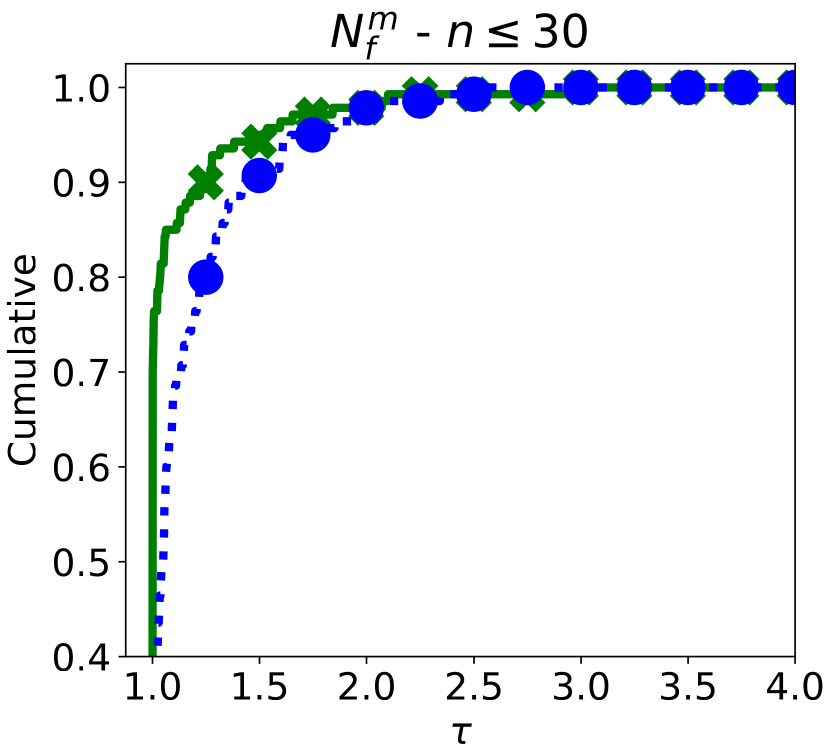}}
		\hfil
		\subfloat{\includegraphics[width=0.25\textwidth]{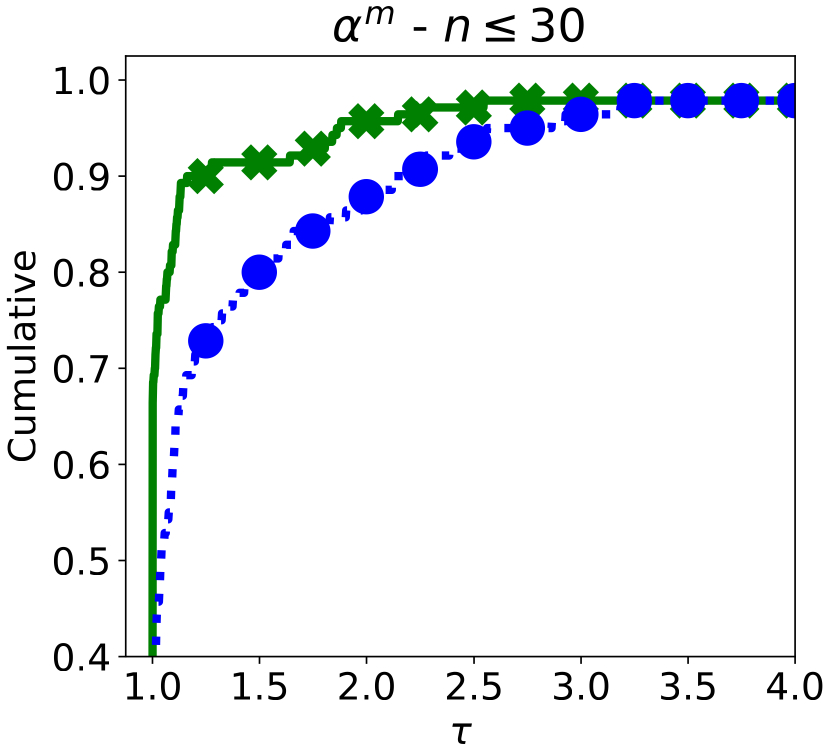}}
		\\
		\subfloat{\includegraphics[width=0.25\textwidth]{Purity_high.pdf}}
		\hfil
		\subfloat{\includegraphics[width=0.25\textwidth]{HV_high.pdf}}
		\hfil
		\subfloat{\includegraphics[width=0.25\textwidth]{Nf_mean_high.pdf}}
		\hfil
		\subfloat{\includegraphics[width=0.25\textwidth]{Alpha_mean_high.pdf}}
		\caption{Performance profiles for \texttt{FPD} and \texttt{FPD\_NMT} w.r.t.\ purity, hyper-volume, $N_f^m$ and $\alpha^m$. First row: full problem benchmark; second row: low-dimensional problems ($n \le 30$); third row: high-dimensional problems ($n > 30$).}
		\label{figapp:all}
	\end{figure*}
}

\section*{Funding}
No funding was received for conducting this study.

\section*{Conflict of interest}
The author declares that he has no conflict of interest.

\section*{Code Availability Statement}
The source code of the \texttt{FPD\_NMT} algorithm can be
found at \href{https://github.com/pierlumanzu/fpd_nmt}{github.com/pierlumanzu/fpd\_nmt}.

\section*{Data Availability Statement}
Data sharing not applicable to this article
as no datasets were generated or analyzed during the current study.

\bibliographystyle{abbrv}

\end{document}